\documentclass[10pt]{amsart}
\usepackage{threeparttable}
\usepackage{mathtools}
\usepackage[ansinew]{inputenc}\usepackage[T1]{fontenc}
\usepackage[all,cmtip]{xy}
\usepackage{tikz} 
\usetikzlibrary{shapes.geometric, arrows}  
\usepackage{geometry}
\usepackage{amscd,amssymb,verbatim,xcolor}
\usepackage{longtable, multirow}
\usepackage[all,cmtip]{xy}
\usepackage[colorlinks=true,
            linkcolor=blue,
            citecolor=blue,
            urlcolor=blue]{hyperref}

\date{\today}

\newcommand\bb[1]{{\mathbb     #1}}
\newcommand\C[1]{{\mathcal #1 }}

\newtheorem{theorem}{Theorem}[section]

\newtheorem{corollary}[theorem]{Corollary}

\newtheorem{definition}[theorem]{Definition}
\newtheorem{example}[theorem]{Example}
\newtheorem{lemma}[theorem]{Lemma}
\newtheorem{proposition}[theorem]{Proposition}
\newtheorem{remark}[theorem]{Remark}

\begin{document}

\setlength{\baselineskip}{1.2\baselineskip}
\title[Shimura correspondence]
{Unitary Shimura correspondence for complex classical groups}

\author[Wan-Yu Tsai]{Wan-Yu Tsai}
\address{Department of Mathematics, National Central University, No. 300, Zhongda Rd., Zhongli District, Taoyuan City 320317, Taiwan}
\email{wytsai@math.ncu.edu.tw}
\author[Kayue Daniel Wong]{Kayue Daniel Wong}

\address{School of Science and Engineering, The Chinese University of Hong Kong (Shenzhen), Longgang, Shenzhen, Guangdong 518172, P. R. China}
\email{kayue.wong@gmail.com}
\author[Hongfeng Zhang]{Hongfeng Zhang}
\address{School of Mathematics and Statistics, Huazhong University of Science and Technology, Wuhan, 430074, P. R. China}
\email{zhanghongf@pku.edu.cn}

\begin{abstract}
In this paper, we construct a lifting operator  from the 
Grothendieck group of admissible Harish-Chandra modules of
$G =\mathrm{SO}_{2n}(\mathbb C)$ (resp. $\mathrm{Sp}_{2n}(\mathbb C)$) to
that of genuine representations of $\mathrm{Spin}_{2n}(\mathbb C)$ (resp. $\mathrm{Spin}_{2n+1}(\mathbb C)$).
We determine the lift of the sum of special unipotent representations attached to any ${}^{\vee}\mathcal O \subseteq {}^{\vee}\mathfrak g$ explicitly.  In particular, the representations
occurring in these lifts, if nonzero, are genuine unipotent representations of
complex Spin groups and are unitary.  
As a consequence, the lifting
operator preserves unitarity on a large class of unitary representations.

\end{abstract}

\maketitle

\section{Introduction}
Let $G$ be a connected real reductive linear group, and let $G'$ be a suitable real form of the Langlands dual of $G$.  Inspired by the work of Shimura on modular forms, one expects a close relation between some stable linear combinations of unitary representations of $G$ and genuine representations of a certain nonlinear cover $\widetilde{G'}$ of $G'$.

More explicitly, let $\mathcal{G}(G)$ denote the Grothendieck group of admissible  finite-length $(\mathfrak{g},K)$-modules  of $G$, and let $\mathcal{G}^{\mathrm{st}}(G)$ be the subgroup consisting of virtual representations whose global characters are stable. Similarly, let 
$\mathcal{G}(\widetilde{G'})$
denote the Grothendieck group of admissible $(\mathfrak{g'},\widetilde{K'})$-modules.  
One would like to construct a map
$$\mathrm{Lift}: \mathcal{G}^{\mathrm{st}}(G) \longrightarrow \mathcal{G}(\widetilde{G'}).$$
The lifting map is defined at the level of global characters: if $\pi\in\mathcal{G}^{\mathrm{st}}(G)$ has character
$\Theta_\pi$, then $\mathrm{Lift}(\pi)$ is characterized by a prescribed
lifting formula for $\Theta_\pi$.  Here a virtual representation of $G$
is called stable if its global character is invariant under
conjugation by the algebraic closure of $G$.
We denote by $\mathcal{G}_{\mathrm{gen}}(\widetilde{G'})\subseteq \mathcal{G} (\widetilde{G'})$ the subgroup generated by 
genuine representations of $\widetilde{G'}$,  that is, those which  {\it do not} descend to a $(\mathfrak{g'},K')$-module. It is an interesting question whether such a lifting operator  maps certain components of the unitary dual of $G$ to the unitary dual of $\widetilde{G'}$, or, more specifically,  to its genuine unitary dual. 

\medskip
This question has been studied in several settings. For  $G = G' = \mathrm{GL}_n(\mathbb{R})$,
Adams and Huang (see \cite{AH97}) constructed an explicit lifting operator, 
where the condition of stability is vacuous. Explicit formulas for this lift are also provided. In particular, 
they proved  that the lift of an irreducible unitary representation is either zero or an irreducible unitary representation.
On the other hand, Adams \cite{A98} and Renard \cite{R98} constructed $\mathrm{Lift}$ between  $G = \mathrm{SO}(n+1,n)$ and $\widetilde{G'} = \mathrm{Mp}_{2n}(\mathbb{R})$ using different approaches.
 They showed that the map $\mathrm{Lift}$ preserves stability and maps stable sums of tempered representations to stable sums of tempered representations. In this setting,  however, it remains unclear in general whether this lifting preserves unitarity.
 
 In a more general setting, Adams and Herb \cite{AH10} established a framework for lifting stable sums of standard representations from \( G \) to \( \widetilde{G'} \) when \( G \) is of simply-laced type, so that \( G \) and \( G' \) share the same adjoint type. In combination with Kazhdan--Lusztig polynomials, character lifting provides a method for computing the lift of certain unipotent representations. For example, the first-named author \cite{T23} computed the lift of the trivial representation for simply-laced, split real semisimple groups. In that case, the lift is the sum of all irreducible genuine small representations of \( \widetilde{G} \) with infinitesimal character \( \rho/2 \), associated to a  specific nilpotent orbit, and characterized by maximal \( \tau \)-invariant. 

As for whether $\mathrm{Lift}$ preserves unitarity or not, there are only partial results. Namely, for $\mathrm{GL}_n(\mathbb{R})$, \cite{AH97} proved that the lift of an irreducible unitary representation is either zero or an irreducible unitary representation, while no general result is known for other groups. 
For the genuine small representations
arising from the lift of the trivial representation in \cite{T23}, those
of types $A$ and $D$ are known to be unitary, while the corresponding
question in type $E$ remains open. In the forthcoming work \cite{Ts}, the first-named author
also studies the lift of the trivial representation for certain nonsplit
real groups of type $D$, where the representations occurring in the lift
are again unitary.

Towards this direction,  \cite{ABPTV07} studied such  a correspondence between some specified $G$ and the double cover of the split real form $G'$.
More precisely, they proved that if $G$ and $G'$ are classical, there is an inclusion between the spherical complementary series of $G$ and the \emph{pseudo}-spherical complementary series of 
$\widetilde{G'}$. In the special case when $G = \mathrm{SO}(n+1,n)$ and $\widetilde{G'} = \mathrm{Mp}_{2n}(\mathbb{R})$, they proved that the inclusion is indeed a bijection. 

\medskip
As for non-archimedean local fields, Kazhdan and Patterson \cite{KP84} studied various aspects of pseudo-spherical representations (called {\it theta-representations}) of some $k$-fold covers of $\mathrm{GL}_n$. More recently, these ideas have been  generalized to other split $p$-adic groups by Gao, Liu, and the first-named author (cf. \cite{GT22} for `small' orbits, \cite{GLT25} for the general case). 
The question of  unitarity
in this setting appears to be much less understood. 
\medskip

In this paper, we generalize the above notion of lifting to a (quasisplit) complex reductive group $G$ treated as a real Lie group. The case of $G = G'= \mathrm{GL}_n(\mathbb{C})$ is studied  by Tadi\'c in \cite{T96}. We consider the following complex classical cases 
$$
G=\mathrm{Sp}_{2n}(\mathbb C),
\quad
\widetilde {G'}=\mathrm{Spin}_{2n+1}(\mathbb C),
$$
and
$$
G=\mathrm{SO}_{2n}(\mathbb C),
\quad
\widetilde {G'}=\mathrm{Spin}_{2n}(\mathbb C).
$$
 Since we work on complex groups, the issue of stability is vacuous, and one only has to define the lift for all virtual characters in the Grothendieck group of $(\mathfrak{g},K)$-modules. 
 
 The complex case provides a particularly accessible setting in which to
investigate the relation between character lifting and unitarity.
Irreducible admissible representations admit an explicit description in
terms of Langlands--Zhelobenko parameters, while special unipotent
representations are governed by the character formulas of Barbasch and
Vogan \cite{BV85}. On the covering-group side, the genuine unipotent
representations of complex Spin groups relevant to our correspondence
have been studied in \cite{Br99,B17,BTs18, WZ23}, and are known to be unitary.
These ingredients allow us to study the lifting operation explicitly and to address the question of unitarity for a substantial class of representations.

 Under our definition of $\mathrm{Lift}$ from $G$ to $\widetilde{G'}$, normalized parabolic induction is preserved.
 This reduces the study of
the lift of special unipotent representations to those attached to
special nilpotent orbits $\mathcal O$ whose Barbasch--Vogan--Lusztig--Spaltenstein (BVLS) dual ${}^{\vee}\mathcal{O}$ is {\it quasi-distinguished} in the sense of \cite{BMSZ25}. 
 For such an orbit $\mathcal O$,
we show that the lift of {\it the sum of} all special unipotent representations
attached to $\mathcal O$ is nonzero if and only if every column of
$\mathcal O$ is divisible by $4$.  
When this condition holds,    the lift is
a sum of genuine unipotent representations of $\widetilde{G'}$ attached
to an explicitly determined nilpotent orbit $\mathcal O'$. These genuine unipotent representations were studied in  \cite{Br99} and \cite{B17} and are known to be unitary. 
More precisely, every special unipotent representation attached to
$\mathcal O$ has the same lift: a single genuine unipotent representation
in the symplectic case, and the sum of two genuine unipotent
representations related by an outer automorphism in the even orthogonal
case.  In particular, together with the compatibility of
$\mathrm{Lift}$ with normalized parabolic induction, this shows that
unitarity is preserved for a large class of unitary representations of
$G$.

\medskip
It is hoped that this manuscript will shed some light on the analogous problems for classical groups over other local fields, and on possible 
extensions of the work of \cite{A98}, \cite{R98} and \cite{ABPTV07}. In these cases, one needs to determine the appropriate stable sums before applying character lifting. In an ongoing work \cite{TW}, we study the lift of stable sums of special unipotent representations for certain simply laced real groups in the framework of \cite{AH10}, which will explain why, in the complex group setting, it is more natural to study the lift of a sum of special unipotent representations rather than that of the individual ones (see Remark \ref{rmk-stable} below). 
It would be interesting to understand to what extent the phenomena established here persist for more general real and $p$-adic groups; in particular, the relation between special pieces, genuine unipotent representations, and preservation of unitarity.

\medskip
The paper is organized as follows. In Section~2, we review character lifting, basics of complex groups, as well as special unipotent representations described in \cite{BV85}, and the description of classical special nilpotent orbits and their Lusztig quotients. In Section~3, we define the genuine lifting map for complex symplectic and even orthogonal groups, prove its compatibility with parabolic induction, compute the lift of special unipotent representations, and state the main theorem. Section~4 is devoted to its proof. We first establish the uniqueness of the relevant genuine unipotent representations using coherent continuation, left cells, Springer representations, and truncated induction, and then apply this uniqueness result to identify the character obtained by lifting.

\section{Preliminaries}
\subsection{Lifting of Representations} \label{sec-lifting}
Let $G$ be a complex linear Lie group treated as a real group, corresponding to the root datum $(X, \Delta, X^{\vee}, \Delta^{\vee})$ (unless stated otherwise, we assume $\mathbb{F} = \mathbb{C}$ and do not specify the local field $\mathbb{F}$ from now on). The Langlands dual group $G'$ of $G$ is a complex reductive group corresponding to the root datum $(X^{\vee}, \Delta^{\vee}, X, \Delta)$. As in \cite[Section 5]{A98} as well as \cite[Section 4]{R98}, we fix Cartan subgroups $H$ and $H'$ of $G$ and $G'$, and an isomorphism 
$$\phi: H \to H'$$
such that as elements of their respective standard representations, the spectra of $h \in H$ and $\phi(h) \in H'$ are identical (cf. \cite[Sections 2--3]{A98}) . For instance, if $G = \mathrm{SO}_{2n}(\mathbb{C}) = G'$ then one can simply take $\phi = \mathrm{id}$. Note that $\phi$ is defined up to inner automorphisms. Also, let
$$\phi^{\vee}: (\mathfrak{h}')^* \to \mathfrak{h}^*$$
by $\phi^{\vee}(\alpha) := 2(d\phi)^*(\alpha)/\langle \alpha,\alpha\rangle$, where $\langle\ , \ \rangle$ is an invariant form on $(\mathfrak{h}')^*$ such that $\langle \beta,\beta \rangle = 2$ for long roots $\beta \in X^*(H')$. Then $\phi^{\vee}$ defines a bijection between the root systems $\Delta$ and $\Delta^{\vee}$ by swapping the long and short roots. This in turn defines an isomorphism
$$\phi^W:W(G,H) \to W(G',H')$$
between the Weyl groups of $G$ and $G'$ such that $\phi^W(\sigma) \circ d\phi = d\phi \circ \sigma$ for all $\sigma \in W(G,H)$.


We now consider the covering group of $G'$. Suppose $G'$ has a finite covering group $\widetilde{G'}$, that is, there exists a nonsplit central extension 
$$1 \to Z \to  \widetilde{G'} \xrightarrow{p} G' \to 1$$ 
for some finite central subgroup $Z$.
Then $\widetilde{H'} := p^{-1}(H')$ is a Cartan subgroup of $\widetilde{G'}$. 

\medskip
The {\it lift} of a stable virtual representation $\pi \in \mathcal{G}^{\mathrm{st}}(G)$ can now be `defined' as follows: suppose $\Theta_{\pi}$ is its global character, treated as a conjugate-invariant function on the set of regular semisimple elements of $G$, then $\Theta_{\pi}$ is determined by its values on the set of regular elements of $H$. The lift of $\pi$ has character formula given by:
\begin{equation} \label{eq:lift}
\mathrm{Lift}(\Theta_{\pi})(\widetilde{h}') := \sum_{\phi(h)^k = p(\widetilde{h}') }\Delta(h,\widetilde{h}') \Theta_{\pi}(h),
\end{equation}
for some specified {\it transfer factor} $\Delta(h,\widetilde{h}')$  satisfying certain properties (see \cite[Equation (1.3)]{AH10} for instance). Let
$$\mathrm{Lift}(\pi) \in \mathcal{G}(\widetilde{G'})$$ 
be the virtual representation whose character formula equals $\mathrm{Lift}(\Theta_{\pi})$ in Equation \eqref{eq:lift}.

\medskip
The transfer factors $\Delta(h,\widetilde{h}')$ are given for various $G$ and $G'$. For instance, in \cite{T96}, Tadi\'c considers the case when $G = G' = \mathrm{GL}_{n}(\mathbb{C})$, $\Phi: H \to H'$ is the identity map, and $\widetilde{G'} = \widetilde{\mathrm{GL}}_{n}(\mathbb{C})$ is the $k$-fold cover of $G'$. Up to a twist by the central character of $G$, one can take $\Delta(h,\widetilde{h}') := \frac{1}{k^{n-1}}\frac{|D(h)|}{|D(\widetilde{h'})|}$, where $D=\prod_{\alpha \in \Delta^+(\mathfrak{g},\mathfrak{h})}(e^{\frac{\alpha}{2}}-e^{-\frac{\alpha}{2}})$ is the Weyl denominator. In particular, 
the resulting lift is a {\it non-genuine} representation of $\widetilde{G'}$. 

\medskip
As for {\it genuine} lifts, we focus on the special case when $k = 2$. For $\mathrm{GL}_{n}(\mathbb{C})$, such lift is defined in \cite[Section 3]{AH97} (see also Proposition \ref{prop:over2} below). Indeed, the main focus of \cite{AH97} is on $G = \mathrm{GL}_{n}(\mathbb{R})$, where one considers a double cover of  $G' = G$. When $G$ is a real form of a connected, simply laced complex group, \cite[Theorem~19.1 and Corollary~19.11]{AH10} show that 
\eqref{eq:lift} is the character of a genuine virtual 
representation of $\widetilde{G}=\widetilde{G'}$, or 0.

In Section \ref{sec:liftclassical} below, we will study the genuine lift of representations for $G = \mathrm{Sp}_{2n}(\mathbb{C})$ and $\mathrm{SO}_{2n}(\mathbb{C})$, with $\widetilde{G'} = \mathrm{Spin}_{2n+1}(\mathbb{C})$ and $\mathrm{Spin}_{2n}(\mathbb{C})$ being the double cover of $G' = \mathrm{SO}_{2n+1}(\mathbb{C})$ and $\mathrm{SO}_{2n}(\mathbb{C})$ respectively (the case of non-genuine lift is easier, and will be considered elsewhere). 

\subsection{Basics on Complex Lie Groups}
To begin with, we recall the classification of irreducible admissible representations of complex Lie groups, see \cite{D75} for more details.

\medskip
Let $G$ be a connected reductive complex Lie group viewed as a real Lie
group. Fix a maximal compact subgroup $K$, and a pair $(B,H=TA)$
where $B$ is a real Borel subgroup and $H$ is a $\theta$-stable Cartan
subgroup in $B$ such that $T=K\cap H$ and $A$ is the complement stabilized by $\theta$.
We write $\mathfrak{g}$, $\mathfrak{k}$, $\mathfrak{b}$, $\mathfrak{h}$,
$\mathfrak{t}$, $\mathfrak{a}$ as their respective Lie algebras, and their
complexifications are denoted by adding the subscript $\bullet_{\mathbb{C}}$. Let $W$ denote the Weyl group $W(\mathfrak{g}, \mathfrak{h})$.

\medskip
The {\bf Langlands--Zhelobenko parameter} of any irreducible module is a pair $(\lambda_L;\lambda_R) \in \mathfrak{h}_{\mathbb{C}}^* \cong \mathfrak{h}^* \times \mathfrak{h}^*$ such that $\mu:=\lambda_L-\lambda_R$ is the parameter of a character of $T$ in the
  decomposition of the $\theta$-stable
  Cartan subalgebra $H=T\cdot A,$ and $\nu:=\lambda_L+\lambda_R$ the
  $A$-character. Then the \textit{principal series representation}
  associated to $(\lambda_L;\lambda_R)$ is the $(\mathfrak{g}, K)$-module
$$
X(\lambda_L;\lambda_R) = X(\mu,\nu) = \mathrm{Ind}_B^G(e^{\mu}\otimes e^\nu \otimes
1)_{K\text{-finite}}
$$
with infinitesimal character $(\lambda_L;\lambda_R)$
(here the symbol $\mathrm{Ind}$ refers to normalized Harish-Chandra induction).

Let $J(\mu,\nu) = J(\lambda_L;\lambda_R)$ be the unique irreducible subquotient of
$X(\mu,\nu) = X(\lambda_L;\lambda_R)$ containing the lowest $K$-type with extremal weight
$\mu=\lambda_L-\lambda_R$. This is called the {\bf Langlands subquotient}.

\begin{proposition}[Parthasarathy--Rao--Varadarajan,  Zhelobenko] \label{prop:BVprop}
Assume $G$ is a complex connected reductive group viewed as a real group, and
let $(\lambda_L;\lambda_R)$ and $(\lambda_L';\lambda_R') $ be parameters described above. Then the
following are equivalent:
\begin{itemize}
\item $(\lambda_L;\lambda_R)=(w\lambda_L';w\lambda_R')$ for some $w\in W.$
\item  $X(\lambda_L;\lambda_R)$ and $X(\lambda_L'; \lambda_R')$ have
  the same composition factors with the same multiplicities.
\item The Langlands subquotient of $X(\lambda_L;\lambda_R)$, written as
$J(\lambda_L;\lambda_R)$, is the same as that of $X (\lambda_L'; \lambda_R')$.
\end{itemize}
Furthermore, every irreducible $(\mathfrak{g}_{\mathbb{C}}, K_{\mathbb{C}})$-module is equivalent to some $J(\lambda_L;\lambda_R)$.
\end{proposition}

\subsection{Special Unipotent representations} \label{sec:specialunip}
In the 1980's, Arthur conjectured that a certain class of automorphic representations of a reductive group $G$ over a local field $F$ should
be associated to $^\vee G$-equivalence classes of homomorphisms
$$
\Phi:\C W_F\times \mathrm{SL}_2(\mathbb{C})\longrightarrow\ ^\vee G
$$
where $\C W_F$ is the Weil group, and $^\vee G$ is the Langlands dual group of $G$. For each equivalence class of homomorphisms, there is a packet of representations (called \textit{Arthur packet}) attached to it, and they are expected to satisfy certain conditions, such as stability and twisted endoscopy.

Special attention is given to the case when $\Phi\mid_{\C W_F}= \mathrm{triv}$, so that the representations in the Arthur packet corresponding to it have infinitesimal character equal to $^\vee h/2$,
where  $\{ ^\vee e,\ ^\vee h,\ ^\vee f\}$ is a Lie triple associated to $\Phi(\mathrm{SL}_2(\mathbb{C}))$. Suppose $^\vee e$ lies in the nilpotent orbit $^\vee\mathcal{O} \subseteq ^\vee\mathfrak{g}$, then the Arthur packet corresponding to $\Phi$ is called a {\it weakly unipotent packet}.

In this paper, we focus on the case when $G$ is a complex reductive Lie group (i.e. $F = \mathbb{C}$), and $^\vee \mathcal{O} \subseteq {}^\vee\mathfrak{g}$ is a special nilpotent orbit (in the sense of Lusztig). Then one has \begin{equation} \label{eq-BVLS}
\mathcal{O} \subseteq \mathfrak{g}  \longleftrightarrow {}^\vee\mathcal{O} \subseteq {}^\vee\mathfrak{g}
\end{equation}
under the {\it Barbasch--Vogan--Lusztig--Spaltenstein (BVLS) duality} among the special nilpotent orbits of $\mathfrak{g}$ and ${}^\vee\mathfrak{g}$. In such a case, the packet corresponding to $\Phi$ in the above paragraph is called the {\it special unipotent packet} $\mathcal{U}(\mathcal{O})$ of $G$ attached to the special nilpotent orbit $\mathcal{O} \subseteq \mathfrak{g}$ (from now on, we will only work on the group side $G$ instead of ${}^\vee G$). 

These representations are well-studied by Barbasch and Vogan \cite{BV85}. We record some of the results therein.

For an ideal $I\subseteq U(\mathfrak g)$, let
$\operatorname{gr}(I)\subseteq S(\mathfrak g)$ denote its associated
graded ideal with respect to the usual filtration on $U(\mathfrak g)$.
We write
$$
\mathcal V(I):=\mathrm{V}(\operatorname{gr}(I))\subseteq\mathfrak g^*
$$
for its annihilator variety, and identify $\mathfrak g^*$ with
$\mathfrak g$ via the Killing form.

\begin{theorem}[Barbasch--Vogan] \label{thm:BV}
Let $G$ be a complex linear Lie group, and $\mathcal{O} \subseteq \mathfrak{g}$ be a special nilpotent orbit in the sense of Lusztig. For each conjugacy class $I \subseteq \overline{A}(\mathcal{O})$ of the Lusztig quotient group, there is a corresponding $\sigma_I \in \widehat{W}$ attached to it by \cite[Theorem 4.7]{BV85} (see Proposition \ref{prop:barao} below). Let
$$R_I := \frac{1}{|W_{^\vee h}|}\sum_{w \in W} \mathrm{tr}_{\sigma_I}(w) X(^\vee h/2; w(^\vee h/2)),$$
where ${}^\vee h$ is the semisimple element of the Jacobson--Morozov triple of the Lusztig--Spaltenstein dual of $\mathcal{O}$. Then the following holds: 

\begin{itemize}
    \item[(a)] The irreducible representations in the special unipotent packet are parametrized by 
$$\mathcal{U}(\mathcal{O}) = \{X_{\sigma}\ |\ \sigma \in \overline{A}(\mathcal{O})^{\wedge}\}.$$
 \item[(b)] For each $X_{\sigma}$ appearing in (a), its character formula is given by
$$X_{\sigma} = \frac{1}{|\overline{A}(\mathcal{O})|}\sum_{x \in \overline{A}(\mathcal{O})} \mathrm{tr}_{\sigma}(x) R_{I_x},$$
where $I_x$ is the (unique) conjugacy class of $\overline{A}(\mathcal{O})$ containing $x$. 
  \item[(c)] Conversely, one has
$$R_{I} = \sum_{\sigma \in \overline{A}(\mathcal{O})^{\wedge}} \mathrm{tr}_{\sigma}(x) X_{\sigma}$$ 
for any $x \in I$. In particular, if $I = \{e\}$ is the trivial conjugacy class, then 
$$R_{\{e\}} = \bigoplus_{\sigma \in \overline{A}(\mathcal{O})^{\wedge}} (X_{\sigma})^{\oplus \dim(\sigma)}.$$
    \item[(d)] Consider the $U(\mathfrak{g}_{\mathbb{C}}) = U(\mathfrak{g}) \otimes U(\mathfrak{g})$-action on $X_{\sigma}$. Then its left and right annihilator ideals $\mathrm{LAnn}(X_{\sigma})$, $\mathrm{RAnn}(X_{\sigma}) \subseteq U(\mathfrak{g})$ satisfy:
$$\mathcal V(\mathrm{LAnn}(X_{\sigma})) =\mathcal{V} (\mathrm{RAnn}(X_{\sigma})) = \overline{\mathcal{O}}.$$

\end{itemize}
\end{theorem}

\subsection{Classical Groups}
In this paper, we are only interested in the case when $G$ is classical. In such a case, the Lusztig quotient group $\overline{A}(\mathcal{O})$ is isomorphic to $(\mathbb{Z}/2\mathbb{Z})^q$ for some nonnegative integer $q$. In particular, it is abelian and all its conjugacy classes are singletons. To better describe the special unipotent packet, we need an explicit description of $\mathcal{O}$.

The classification of special nilpotent orbits, along with their Lusztig quotient $\overline{A}(\mathcal{O})$ in terms of the {\bf columns} (or {\bf dual partition}) of the Young diagrams attached to $\mathcal{O}$ are given as follows:
\begin{proposition}[\cite{W16}, Propositions 2.3--2.4] \label{prop:class} 
The classification of classical special nilpotent orbits in terms of columns is given as follows:

\noindent {\bf Type $B_n$:} Let $G = \mathrm{SO}_{2n+1}(\bb{C})$, then all nilpotent orbits in $\mathfrak{g}$ are parametrized by Young diagrams $\mathcal{O} = (a_{2k+1} \geq a_{2k} \geq \dots \geq a_1 \geq a_0)$ of size $2n+1$ such that $a_{2l}+a_{2l-1}$ is even for all $l$ (We insist that there is an even number of columns, by taking $a_0 = 0$ if necessary).\\
The orbit $\mathcal{O}$ is special if all non-zero $a_i$'s are odd, or the columns of even sizes occur only in the form $a_{2l} = a_{2l-1} = 2b$ (by convention, we take $a_{-1} = 0$). In particular, one must have $a_0 = 0$.\\

\noindent {\bf Type $C_n$:} Let $G = \mathrm{Sp}_{2n}(\bb{C})$, then all nilpotent orbits in $\mathfrak{g}$ are parametrized by Young diagrams $\mathcal{O} = (a_{2k} \geq a_{2k-1} \geq \dots \geq a_1  \geq a_0)$ of size $2n$ such that $a_{2l}+a_{2l-1}$ is even for all $l$ (We insist that there is an odd number of columns, by taking $a_0 = 0$ if necessary). \\
The orbit $\mathcal{O}$ is special if all non-zero $a_i$'s are even, or the columns of odd sizes occur only in the form $a_{2l} = a_{2l-1} = 2c+1$.\\

\noindent {\bf Type $D_n$:} Let $G = \mathrm{SO}_{2n}(\bb{C})$, then all nilpotent orbits in $\mathfrak{g}$ are parametrized by Young diagrams $\mathcal{O} = (a_{2k+1} \geq a_{2k} \geq \dots \geq a_1  \geq a_0)$ of size $2n$ such that $a_{2l}+a_{2l-1}$ is even for all $l$ (We insist that there is an even number of columns, by taking $a_0 = 0$ if necessary). The only exceptions are the following - if the Young diagram is of the form
$$(2\alpha_k, 2\alpha_k, 2\alpha_{k-1}, 2\alpha_{k-1},\dots, 2\alpha_1, 2\alpha_1),$$
i.e. the diagram is {\bf very even}, then there are two orbits $\mathcal{O}_{\mathrm{I}}$, $\mathcal{O}_{\mathrm{II}}$ attached to this diagram. These orbits are called {\bf very even orbits}.\\
The orbit $\mathcal{O}$ is special if all non-zero $a_i$'s are even, or the columns of odd sizes occur only in the form $a_{2l} = a_{2l-1} = 2d+1$. In particular, all very even orbits are special.

\bigskip
Then the Lusztig quotient $\overline{A}(\mathcal{O})$ of the classical special nilpotent orbits $\mathcal{O} \subseteq {}\mathfrak{g}$ in terms of columns are given as follows:

\noindent {\bf Type $B_n$:} Let $G = \mathrm{SO}_{2n+1}(\bb{C})$ and $\mathcal{O} = (a_{2k+1} \geq a_{2k} \geq \dots \geq a_1)$ be a special orbit (note that $a_0 = 0$ by the above discussions and is omitted). Separate all column pairs $a_{2m+1} = a_{2m} = \nu$ and all {\bf even} column pairs $a_{2l} = a_{2l-1} = \mu$ and get
\begin{align*}
    \mathcal{O} = {}\mathcal{O}'' \sqcup {}\mathcal{O}_{\mu} \sqcup {}\mathcal{O}_{\nu} := &(a_{2q+1}'' > a_{2q}'' \geq a_{2q-1}'' > a_{2q-2}'' \geq a_{2q-3}'' > \dots > a_2'' \geq a_1'') \\ &\sqcup (\mu_1, \mu_1, \dots, \mu_x, \mu_x) \sqcup (\nu_1, \nu_1, \dots, \nu_y, \nu_y)
\end{align*}
then $\overline{A}(\mathcal{O}) \cong \overline{A}(\mathcal{O}'') \cong (\bb{Z}/2\bb{Z})^q$.\\ 
		
\noindent {\bf Type $C_n$:} Let $G = \mathrm{Sp}_{2n}(\bb{C})$ and $\mathcal{O} = (a_{2k} \geq a_{2k-1} \geq \dots \geq a_1  \geq a_0)$ be a special orbit. Separate all column pairs $a_{2m+1} = a_{2m} = \nu$ and all {\bf odd} column pairs $a_{2l} = a_{2l-1} = \mu$ and get
    \begin{align*}
    \mathcal{O} = {}\mathcal{O}'' \sqcup {}\mathcal{O}_{\mu} \sqcup {}\mathcal{O}_{\nu} := &(a_{2q}'' \geq a_{2q-1}'' > a_{2q-2}'' \geq a_{2q-3}'' > \dots > a_2'' \geq a_1''  > a_0'')\\ &\sqcup (\mu_1, \mu_1, \dots, \mu_x, \mu_x) \sqcup (\nu_1, \nu_1, \dots, \nu_y, \nu_y) 
    \end{align*}
    then $\overline{A}(\mathcal{O}) \cong \overline{A}(\mathcal{O}'')\cong (\bb{Z}/2\bb{Z})^q$.\\ 

\noindent {\bf Type $D_n$:} Let $G = \mathrm{SO}_{2n}(\bb{C})$ and $\mathcal{O} = (a_{2k+1} \geq a_{2k} \geq \dots \geq a_1  \geq a_0)$ be a special, non-very even orbit. Separate all column pairs $a_{2m+1} = a_{2m} = \nu$ and  all {\bf odd} column pairs $a_{2l} = a_{2l-1} = \mu$ and get
    \begin{align*}
    \mathcal{O} = {}\mathcal{O}'' \sqcup {}\mathcal{O}_{\mu} \sqcup {}\mathcal{O}_{\nu} := &(a_{2q+1}'' > a_{2q}'' \geq a_{2q-1}'' > a_{2q-2}'' \geq a_{2q-3}'' > \dots > a_2'' \geq a_1'' > a_0'') \\ &\sqcup (\mu_1, \mu_1, \dots, \mu_x, \mu_x) \sqcup (\nu_1, \nu_1, \dots, \nu_y, \nu_y)
    \end{align*}
    then $\overline{A}( \mathcal{O}) \cong \overline{A}(\mathcal{O}'') \cong (\bb{Z}/2\bb{Z})^q$. Moreover, if $\mathcal{O}_{\mathrm{I}}$ and $\mathcal{O}_{\mathrm{II}}$ are very even, then  $\overline{A}(\mathcal{O}_{\mathrm{I}})$ and $\overline{A}(\mathcal{O}_{\mathrm{II}})$ are trivial.
\end{proposition}

\begin{remark} \label{rmk-bmszdual}
    By the description of $\mathcal{O}$ above, one can compute the BVLS dual ${}^{\vee}\mathcal{O}$ in Equation \eqref{eq-BVLS} as follows: write $\mathbf{N} := [\mu_1, \mu_1, \dots, \mu_x, \mu_x] \sqcup [\nu_1, \nu_1, \dots, \nu_y, \nu_y]$ in all the above cases, then in terms of {\bf row sizes} of Jordan blocks, ${}^{\vee}\mathcal{O}$ is given by:
    \begin{itemize}
        \item{\bf Type $B_n$:} ${}^{\vee}\mathcal{O} = [a_{2q+1}''-1,  a_{2q}''+1, a_{2q-1}''-1, a_{2q-2}''+1, a_{2q-3}''-1, \dots, a_2''+1, a_1''-1] \sqcup \mathbf{N}$;
        \item{\bf Type $C_n$:} ${}^{\vee}\mathcal{O} = [a_{2q}''+1,  a_{2q-1}''-1, \dots, a_2''+1, a_1''-1, a_0''+1] \sqcup \mathbf{N}$;
        \item{\bf Type $D_n$:} ${}^{\vee}\mathcal{O} = [a_{2q+1}''-1, a_{2q}''+1,  a_{2q-1}''-1, \dots, a_2''+1, a_1''-1, a_0''+1]\sqcup \mathbf{N}$.
    \end{itemize}
This description is compatible with that in \cite[Section 2.3--2.8]{BMSZ25} where, in their formulations, our $\mu_i$'s and $\nu_i$'s in $\mathbf{N}$ consist of rows with {\bf bad parity} or {\bf balanced rows} \cite[Definition 2.21]{BMSZ25}, and the orbit ${}^{\vee}\mathcal{O}$ is {\bf quasi-distinguished} \cite[Section 10.1]{BMSZ25} precisely when $\mathbf{N} =
\emptyset$. Therefore, our description of $\overline{A}(\mathcal{O}) \cong \overline{A}({}^{\vee}\mathcal{O})$ is compatible with that in \cite[Remark 2.22]{BMSZ25} and \cite{S01}.
\end{remark}

Given the description of nilpotent orbits of $G$ above, one has the following:
\begin{theorem} \label{thm:reduce} \textup{(\cite[Section~2.9]{BMSZ25} and \cite[Proposition 8.10]{BV85})}
    Let $\mathcal{O}$ be a special classical nilpotent orbit, and $ \mathcal{O}'' \sqcup  \mathcal{O}_{\mu} \sqcup  \mathcal{O}_{\nu}$ be as given above such that ${}^{\vee}\mathcal{O}''$ is quasi-distinguished. Then the special unipotent representations attached to $ \mathcal{O}$ is given by
    \[\mathcal{U}( \mathcal{O}) = \left\{ \prod_{i = 1}^x \mathrm{triv}_{\mathrm{GL}_{\mu_i}} \times \prod_{j = 1}^y \mathrm{triv}_{\mathrm{GL}_{\nu_j}} \rtimes   \pi \ \Bigg|\ \pi \in \mathcal{U}(\mathcal{O}'') \right\},\]
    where $\prod_i \chi_i \times \prod_j \lambda_j \rtimes \pi := \mathrm{Ind}_P^G\left(\bigotimes_i \chi_i \otimes \bigotimes_j \lambda_j\otimes \pi \otimes {\bf 1} \right)$
    is the normalized induction from the parabolic subgroup $P = MN$ with 
     \begin{itemize}
         \item Levi subgroup $M := \prod_{i} \mathrm{GL}_{\mu_i}(\mathbb{C}) \times \prod_{j} \mathrm{GL}_{\nu_j}(\mathbb{C})  \times G''$, where $G''$ is of the same Lie type as $G$;
         \item $\chi_i$ is a representation of $\mathrm{GL}_{\mu_i}(\mathbb{C})$, $\lambda_j$ is a representation of $\mathrm{GL}_{\nu_j}(\mathbb{C})$, $\pi$ is a representation of $G''$; and 
         \item $\bf 1$ is the trivial representation of the unipotent radical $N$.
         \end{itemize}
\end{theorem}

In view of the above theorem, we are reduced to studying the special unipotent representations attached to $ \mathcal{O} = \mathcal{O}''$. Unless otherwise specified, we will only focus on studying such orbits from now on. 

\begin{remark}\label{r:infchar}
    Suppose $G$ is of type $C$ or $D$. 
    For later use, we give a direct description of the infinitesimal
character ${}^\vee h/2$ in terms of the column partition of
$\mathcal O=\mathcal O''$.  Write
$r_i=a_i''/2$.  Then one can read off directly from the formulas in Remark \ref{rmk-bmszdual} that  
\begin{equation} \label{e:infchar}
\frac{{}^\vee h}{2}
=\bigsqcup_i
(r_{2i},r_{2i}-1,\ldots,1)
\sqcup
\bigsqcup_{j \geq 1}
(r_{2j-1}-1,r_{2j-1}-2,\ldots,0).
\end{equation}
up to Weyl conjugacy. Thus, \eqref{e:infchar} gives the common infinitesimal character of the representations in $\mathcal U(\mathcal O)$. 
\end{remark}

As mentioned in the paragraph after Theorem \ref{thm:BV}, since $\overline{A}(\mathcal{O}) \cong (\mathbb{Z}/2\mathbb{Z})^q$ is abelian, the $2^q$ conjugacy classes $I \subseteq \overline{A}( \mathcal{O})$ are singletons. Using \cite[Theorem 0.4]{L97} and \cite[Theorem 4.7(c)]{BV85}, the $\sigma_I \in \widehat{W}$ appearing in Theorem \ref{thm:BV} can be given by:

\begin{proposition} \label{prop:barao}
Let $\mathcal{O} =\mathcal{O}''$ be a special nilpotent orbit of classical type as given in Proposition \ref{prop:class}, and $\mathcal{P}(q)$ be the power set of $\{1,2, \dots, q\}$ (so that $|\mathcal{P}(q)| = 2^q$). For each $I \in \mathcal{P}(q)$, consider the nilpotent orbit
$$ \mathcal{O}_I := \begin{cases} (b_{2q} \geq b_{2q-1} \geq \dots \geq b_2 \geq b_1 \geq b_0) &\text{if }  \mathcal{O} \text{ is of Type } C \\
(b_{2q+1} \geq b_{2q} \geq b_{2q-1} \geq \dots \geq b_2 \geq b_1 \geq b_0)  &\text{if }  \mathcal{O} \text{ is of Type } B, D  \end{cases},$$
where 
$$(b_{2i}, b_{2i-1}) := \begin{cases} (a_{2i}''+1, a_{2i-1}'' -1) &\text{if } i \in I \\
(a_{2i}'', a_{2i-1}'') &\text{otherwise} \end{cases}$$
(in particular, $ \mathcal{O}_{\emptyset} =  \mathcal{O}$). Under a suitable choice of the bijection between $\mathcal{P}(q)$ and the conjugacy classes of $\overline{A}( \mathcal{O}) \cong (\mathbb{Z}/2\mathbb{Z})^q$ satisfying $\emptyset\leftrightarrow \{e\}$ (by abuse of notation, we use $I$ to denote the elements on both sides of the bijection), the $\sigma_I \in \widehat{W}$ in Theorem \ref{thm:BV} are given by
$$\sigma_I = \mathrm{Spr}( \mathcal{O}_I),$$
where $\mathrm{Spr}: \mathrm{Nil}( \mathfrak{g}) \hookrightarrow \widehat{W}$ is the Springer representation corresponding to any nilpotent orbit 
$ \mathcal{O} \in \mathrm{Nil}( \mathfrak{g})$.
\end{proposition}

\begin{example}[\cite{W16} Section 3] \label{eg:charform}
    Let $\mathcal{O} = (8,6,4,2,0)$ in $G = \mathrm{Sp}_{20}(\mathbb{C})$. Then 
    \begin{itemize}
        \item $\overline{A}(\mathcal{O}) \cong (\mathbb{Z}/2\mathbb{Z})^2$.
        \item $^\vee \mathcal{O} = (5,3,3,3,3,1,1,1,1)$ and hence ${}^{\vee}h/2 = (4,3,2,1; 2,1,0; 2,1; 0).$
        \item $\mathcal{O}_{\emptyset} = (8,6,4,2,0), \quad \mathcal{O}_{\{2\}} = (9,5,4,2,0), \quad \mathcal{O}_{\{1\}} = (8,6,5,1,0), \quad \mathcal{O}_{\{1,2\}} = (9,5,5,1,0)$.
    \end{itemize}
and hence
$$\sigma_{\emptyset} = 
\begin{pmatrix}
0 &   & 1 &   & 3 &   & 4 &   & 6 \\
  & 1 &   & 2 &   & 4 &   & 5 &  
\end{pmatrix} =
31 \times 42 = j_{W(D_3 \times D_1 \times B_4 \times B_2)}^{W(B_{10})}(\mathrm{sgn}),$$ 
$$\sigma_{\{2\}} = 
\begin{pmatrix}
1 &   & 2 &   & 3 &   & 4 &   & 6 \\
  & 0 &   & 1 &   & 4 &   & 5 &  
\end{pmatrix}= 51 \times 22 = j_{W(D_5 \times D_1 \times B_2 \times B_2)}^{W(B_{10})}(\mathrm{sgn})$$
$$\sigma_{\{1\}} = 
\begin{pmatrix}
0 &   & 1 &   & 4 &   & 5 &   & 6 \\
  & 1 &   & 2 &   & 3 &   & 4 &  
\end{pmatrix} =
33 \times 4  = j_{W(D_3 \times D_3 \times B_4)}^{W(B_{10})}(\mathrm{sgn})$$  
$$\sigma_{\{1,2\}} = 
\begin{pmatrix}
1 &   & 2 &   & 4 &   & 5 &   & 6 \\
  & 0 &   & 1 &   & 3 &   & 4 &  
\end{pmatrix}= 53 \times 2 = j_{W(D_5 \times D_3 \times B_2)}^{W(B_{10})}(\mathrm{sgn}).$$
where $j_{W_0}^{W}(\cdot)$ is the truncated induction (cf. see \cite[Definition 4.11]{BV85} for instance), and the representations of $(\lambda \times \rho) \in W(B_n)^{\wedge}$ are denoted by a pair of partitions $\lambda$ and $\rho$ (labeled by their column sizes) with $|\lambda|+|\rho|= n$. By \cite[Proposition 6.6]{BV85}, one therefore has

\begin{align*}
R_{\emptyset} &= \sum_{w \in W(B_4 \times D_3 \times B_2 \times D_1)} \mathrm{sgn}(w)X\left(\begin{matrix} \quad 4321; &210; &21; &0)\\ w(4321;&210;&21;&0) \end{matrix}\right);\\
R_{\{2\}} &= \sum_{w \in W(D_5 \times B_2 \times B_2 \times D_1)} \mathrm{sgn}(w)X\left(\begin{matrix} \quad 43210;&21;&21;&0\\ w(43210;&21;&21;&0)\end{matrix}\right);\\
R_{\{1\}} &= \sum_{w \in W(B_4 \times D_3 \times D_3)} \mathrm{sgn}(w)X\left(\begin{matrix} \quad 4321;&210;&210\\ w(4321;&210;&210)\end{matrix}\right);\\
R_{\{1,2\}} &= \sum_{w \in W(D_5 \times B_2 \times D_3)} \mathrm{sgn}(w)X\left(\begin{matrix} \quad 43210;&21;&210\\ w(43210;&21;&210)\end{matrix}\right).
\end{align*}    
Note that by Theorem \ref{thm:BV}(c), $\displaystyle R_{\emptyset} = \bigoplus_{\pi \in \overline{A}(\mathcal{O})^{\wedge}} X_{\pi}$ is the sum of all special unipotent representations with multiplicity one.
\end{example}

\begin{remark} \label{rmk-stablesum}
    The nilpotent orbits $\mathcal{O}_I$ appearing in Proposition \ref{prop:barao} are precisely the {\bf special pieces} $\mathrm{SP}(\mathcal{O})$ of the nilpotent orbit $\mathcal{O}$. 
    By Theorem \ref{thm:BV}(a) and Proposition \ref{prop:barao} above, the number of special unipotent representations attached to $\mathcal{O}$ is equal to $|\mathrm{SP}(\mathcal{O})|$. 

    In \cite{BT11}, Barbasch and Trapa studied the stable sums of special unipotent representations attached to a special orbit $\mathcal{O}$ for real reductive groups $G_{\mathbb R}$. In the case when $G_{\mathbb R}$ is a complex group, Propositions 3.2-3.3 of loc. cit. imply that all (stable sums of) special unipotent representations are parametrized by elements of $\mathrm{SP}(\mathcal{O})$. This matches with our discussions in the previous paragraph.
    
\end{remark}

\section{Lifting for Complex Classical Groups} \label{sec:liftclassical}
In this section, we define and study the genuine lift
$$\mathrm{Lift}: \mathcal{G}^{\mathrm{st}}(G) \longrightarrow \mathcal{G}_{\mathrm{gen}}(\widetilde{G'})$$
for $G = \mathrm{Sp}_{2n}(\mathbb{C})$ or $\mathrm{SO}_{2n}(\mathbb{C})$, and $\widetilde{G'} = \mathrm{Spin}_{2n+1}(\mathbb{C})$ or $\mathrm{Spin}_{2n}(\mathbb{C})$ respectively. More explicitly, 
recall the isomorphism $\phi: H \xrightarrow{\cong} H'$ between the Cartan subgroups given in the beginning of Section \ref{sec-lifting}. As for the double cover $\widetilde{H}'$, let $\delta \in X^{\vee}_{\mathbb Q} := X^{\vee} \otimes_{\mathbb{Z}} \mathbb Q$ be a weight of (one of) the spinor module(s) of $\widetilde{G'}$, and let $\widetilde{X}^{\vee} = X^{\vee} + \mathbb{Z}\delta$. Then
$$\widetilde{H}' \cong \mathrm{Hom}_{\mathbb{Z}}(\widetilde{X}^{\vee}, \mathbb{C}^*),$$
and $p: \widetilde{H}' \to H'$ is defined by restricting the above homomorphism to $X^{\vee}$. 

For any $\widetilde{h}' \in \widetilde{H}'$ such that $\phi(h)^2 = p(\widetilde{h}')$, one can check that 
$$\widetilde{h}'(\delta)= \epsilon_{h,\widetilde{h}'} \cdot (\phi(h))(2\delta)$$
for some $\epsilon_{h,\widetilde{h}'} \in \{\pm 1\}$, which is independent of the choice of the spinor weight $\delta$. We define the {\it genuine lift} of $\pi \in \mathcal{G}^{\mathrm{st}}(G)$ by:
\begin{equation} \label{eq:liftclassical}
\mathrm{Lift}(\Theta_{\pi})(\widetilde{h}') := \frac{1}{2^n}\sum_{\phi(h)^2 = p(\widetilde{h}')}  \epsilon_{h,\widetilde{h}} \frac{|D_{G}(h)|}{|D_{\widetilde{G'}}(\widetilde{h}')|}\Theta_{\pi}(h),
\end{equation}
where $D_{G}$ (resp. of $D_{\widetilde{G'}}$) are the Weyl denominators of $G$ (resp. of $\widetilde{G'}$), comparing with Equation \eqref{eq:lift} above.


\medskip
With the above definition of $\mathrm{Lift}$, one can immediately see that the following holds (cf. \cite[Equation (3.10)]{AH97}):
\begin{proposition} \label{prop:over2}
    For principal series representations, the genuine lift is given by
$$\mathrm{Lift}(X(\lambda_L,\lambda_R)) =\begin{cases}
    X(\frac{\lambda_L}{2}, \frac{\lambda_R}{2}) & \text{if\ }\lambda_L - \lambda_R \in (2\mathbb{Z}+1)^n\\
    0 & \text{otherwise}
\end{cases}$$
\end{proposition}

\subsection{Lift of parabolic induction}
\label{s:Lift-PI}
Let $P'$ be a parabolic subgroup of $G'=\mathrm{SO}_{2n+1}(\mathbb{C})$ or $\mathrm{SO}_{2n}(\mathbb{C})$,  whose Levi subgroup is denoted by $L'$.  Then $L'\simeq \prod_{i=1}^{m}\mathrm{GL}_{n_i}(\mathbb{C})\times G_r$, with 
\[G_r=\begin{cases}\mathrm{SO}_{2r+1}(\mathbb{C}) , \ & \text{when}\ G'=\mathrm{SO}_{2n+1}(\mathbb{C}),\\  \mathrm{SO}_{2r}(\mathbb{C}),  \ & \text{when}\ G'=\mathrm{SO}_{2n}(\mathbb{C}).\end{cases}\]
The preimage $p^{-1}(P')$ is a parabolic subgroup of $\widetilde{G'}$ with Levi subgroup 
\[p^{-1}(L')\simeq p^{-1}\left(\mathrm{GL}_{n_1}(\mathbb{C})\right)\cdot \ldots \cdot p^{-1}\left(\mathrm{GL}_{n_m}(\mathbb{C})\right)\cdot p^{-1}(G_r),\]
where the factors $\widetilde{\mathrm{GL}}_{n_i}(\mathbb{C}) =p^{-1}\left(\mathrm{GL}_{n_i}(\mathbb{C})\right)$ are connected double covers of $\mathrm{GL}_{n_i}(\mathbb{C})$, and $p^{-1}(G_r)$ is a complex spin group. These factors all contain the subgroup $p^{-1}(1)\simeq \mathbb{Z}/2\mathbb{Z}$. Moreover, these factors commute with each other since they are connected. 

Let $\pi_i$ be a genuine representation of $\widetilde{\mathrm{GL}}_{n_i}(\mathbb{C})$, and let $\pi_{\sharp}$ be a genuine representation of $p^{-1}(G_r)$. Then $\bigotimes_{i=1}^m\pi_i\otimes \pi_{\sharp}$ is a well-defined genuine representation of $p^{-1}(L')$ and, analogous to the notations in Theorem \ref{thm:reduce}, denote the parabolically induced module by:
\[\prod_{i=1}^m\pi_i\rtimes \pi_{\sharp} :=\mathrm{Ind}_{p^{-1}(P')}^{\widetilde{G'}}\left(\bigotimes_{i=1}^m\pi_i\otimes \pi_{\sharp}\otimes {\bf 1}\right).\] 

\begin{proposition}\label{parab_ind}
The lift preserves parabolic induction, that is, $\mathrm{Lift}(\pi_1\times \pi_2)\simeq \mathrm{Lift}(\pi_1)\times \mathrm{Lift}(\pi_2)$. 
\end{proposition}


\begin{proof}
It follows from the character formula of induced representations, and Proposition \ref{prop:over2} above. The case of $G = \mathrm{GL}_{n}(\mathbb{C})$ is as given in \cite[Theorem 4.5]{T96}.
\end{proof}

By Theorem \ref{thm:reduce} and Proposition \ref{parab_ind}, it suffices to study the lift of trivial representations of $\mathrm{GL}_{n}(\mathbb{C})$ and that of the special unipotent representations attached to $\mathcal{O} = \mathcal{O}''$. The former is outlined in \cite{T96}, which we will recall in Section \ref{sec:gln} below (note that Tadi\'c only dealt with the non-genuine lift, so our results will be slightly different from his); As for the latter, we have already seen in Remark \ref{rmk-bmszdual} that these are precisely orbits whose BVLS dual ${}^{\vee}\mathcal O$ is quasi-distinguished. We will give an outline of the lift of representations in $\mathcal{U}(\mathcal{O})$ in Section \ref{sec:liftunip}.

\subsection{Lift of $\mathrm{triv}$ in $\mathrm{GL}_{n}(\mathbb{C})$}
\label{sec:gln}
In this section, we set $G=G' = \mathrm{GL}_{n}(\mathbb{C})$, and set $\widetilde{G'}$ to be the 
connected double
cover of $\mathrm{GL}_n(\mathbb C)$ appearing in Section \ref{s:Lift-PI}. Equivalently, it
may be realized as
$$
\widetilde{\mathrm{GL}}_n(\mathbb C)
=
\left\{(g,z)\in \mathrm{GL}_n(\mathbb C)\times\mathbb C^\times
\mid z^2=\det(g)\right\},
$$
with covering map $p(g,z)=g$.

For $a\in\frac12\mathbb Z$ and $s\in\mathbb C$,  define the following character of $\widetilde{G'}$: 
$$
\left(\frac{\det}{|\det|}\right)^a |\det|^s(g,z)
:=
\left(\frac{z}{|z|}\right)^{2a}|\det(g)|^s.
$$
This character descends to $\mathrm{GL}_n(\mathbb C)$ when $a\in\mathbb Z$,
and is genuine when $a\in\frac12+\mathbb Z$.

\begin{proposition}\label{lift_1_gln}
Let $G = \mathrm{GL}_{n}(\mathbb{C})$. Up to $\pm 1$, the lift of the trivial representation is:

  \[   \mathrm{Lift}(\mathrm{triv}) = \begin{cases} \displaystyle \left(\frac{\det}{|\det|}\right)^{1/2}_{\widetilde{\mathrm{GL}}_m}\times \left(\frac{\det}{|\det|}\right)^{-1/2}_{\widetilde{\mathrm{GL}}_m} &  n=2m;\\
   0, & n=2m+1.\end{cases}\]
\end{proposition}
\begin{proof}
By Weyl character formula, 
\[\mathrm{triv}=\sum_{w\in W}\mathrm{sgn}(w)X(\rho,w\rho),\] 
as virtual representations, where $\rho$ is the half sum of the positive roots. Hence, 
$$\mathrm{Lift}(\mathrm{triv})=\sum_{w\in W}\mathrm{sgn}(w) \mathrm{Lift}(X(\rho,w\rho)).$$ 

When $n=2m$, 
$\rho = \left(\frac{2m-1}{2},\ \frac{2m-1}{2}-1,\ \dots, \frac{2m-1}{2} - (2m-2), \frac{2m-1}{2}-(2m-1)\right)$, and it is easy to see that
$\rho-w\rho \in (2\mathbb{Z}+1)^n$ if and only if 
{\scriptsize $$w\rho = \left(w_1\left(\frac{2m-1}{2}, \frac{2m-1}{2}-2,\dots, \frac{2m-1}{2}-(2m-2)\right); \\
w_2\left(\frac{2m-1}{2}-1, \frac{2m-1}{2}-3,\dots, \frac{2m-1}{2}-(2m-1)\right) \right)$$}
for some $w_1,w_2\in S_m$. As a result, 
\[\begin{aligned}\mathrm{Lift}(\mathrm{triv}) \ &=\sum_{w_1,w_2\in S_m} \mathrm{sgn}((w_1,w_2) \cdot y)X\left(\frac{\rho}{2},\ ((w_1,w_2) \cdot y)\frac{\rho}{2}\right)\\
& \ = \det(y) \left(\frac{\det}{|\det|}\right)^{1/2}_{\widetilde{\mathrm{GL}}_m}\times \left(\frac{\det}{|\det|}\right)^{-1/2}_{\widetilde{\mathrm{GL}}_m}, \end{aligned}\]
where $y \in S_n = S_{2m}$ is given by $y(0\ 1\ \dots\ 2m-1) := (0\ 2\ \dots\ 2m-2; 1\ 3\ \dots\ 2m-1)$ with $\mathrm{sgn}(y) = (-1)^{\lceil m/2 \rceil}$.

When $n=2m+1$, a simple parity argument implies that $\rho-w\rho \notin (2\mathbb{Z}+1)^n$ for any $w\in S_n$. Therefore, the trivial representation cannot be lifted to a genuine representation.  
\end{proof}

\begin{theorem}
    Let $\pi$ be an irreducible unitary representation of $\mathrm{GL}_{n}(\mathbb{C})$. Then (up to $\pm$) $\mathrm{Lift}(\pi)$ is either zero, or an irreducible unitary representation of $\widetilde{\mathrm{GL}}_{n}(\mathbb{C})$.
\end{theorem}
\begin{proof}
By the classification of unitary dual of $\mathrm{GL}_{n}(\mathbb{C})$ and Proposition \ref{parab_ind}, one can reduce the proof to the case when $\pi$ is a unitary character, or a Stein complementary series.  
When $\pi = (\det/|\det|)_{\mathrm{GL}_n}^k$ is a unitary character with $k \in \mathbb{Z}$, then by Proposition \ref{lift_1_gln}, $\mathrm{Lift}(\pi)$ is zero if $k$ is odd, or a product of two unitary characters if $k$ is even. Therefore, it is unitary and irreducible. 

When $\pi = (|\det|_{\mathrm{GL}_{n/2}}^{\nu/2} \times |\det|_{\mathrm{GL}_{n/2}}^{-\nu/2}) \otimes (\det/|\det|)_{\mathrm{GL}_n}^k$ $(0 < \nu < 1)$ is a Stein complementary series (so that $n = 2m$ is even), then by Proposition \ref{parab_ind} and Proposition \ref{lift_1_gln} again, $\mathrm{Lift}(\pi)$ is either zero (which is always the case when $m$ is odd), or equal to $\widetilde{\pi}_+ \times \widetilde{\pi}_-$, where $\widetilde{\pi}_{\pm}$ is a Stein complementary series in $\mathrm{GL}_{m}(\mathbb{C})$ twisted by a genuine character $(\frac{\det}{|\det|})^{\pm 1/2}$ in $\widetilde{\mathrm{GL}}_m$. By the classification of the unitary dual of $\mathrm{GL}_n$ in \cite{V86}, this is irreducible and the theorem follows as a consequence.
\end{proof}

\subsection{Lift of Special Unipotent Representations for Classical Groups} \label{sec:liftunip}
We now move on to studying the special unipotent representations $\mathcal{U}(\mathcal{O})$ attached to nilpotent orbits $\mathcal{O} = \mathcal{O}''$ whose dual orbit ${}^{\vee}\mathcal{O}$ is quasi-distinguished when $G = \mathrm{Sp}_{2n}(\mathbb{C})$ or $\mathrm{SO}_{2n}(\mathbb{C})$. In view of the character formula in Theorem \ref{thm:BV}(c), we begin by understanding the lift of the sum of all special unipotent representations using the notations in Proposition \ref{prop:barao}:
$$R_{\{e\}} = R_{\emptyset} = \bigoplus_{\pi \in \overline{A}(\mathcal{O})^{\wedge}} X_{\pi}.$$  

\medskip
We obtain the following lifting formulas for $R_\emptyset$ in the following proposition. 
\begin{proposition} \label{prop:lift}
    Let $G = \mathrm{Sp}_{2n}(\mathbb{C})$ or $G = \mathrm{SO}_{2n}(\mathbb{C})$, and 
    $$\mathcal{O} = \begin{cases} (a_{2q}'' \geq a_{2q-1}'' > \dots > a_2'' \geq a_1'' > a_0'')   &\text{if }  G = \mathrm{Sp}_{2n}(\mathbb{C}) \\
(a_{2q+1}'' > a_{2q}'' \geq a_{2q-1}'' > \dots > a_2'' \geq a_1'' > a_0'') &\text{if }  G = \mathrm{SO}_{2n}(\mathbb{C}) \end{cases}$$ 
    be a special nilpotent orbit as given in Proposition \ref{prop:class}. Then the lift of the sum of special unipotent representations in $\mathcal{U}(\mathcal{O})$ to $\widetilde{G'} = \mathrm{Spin}_{2n+1}(\mathbb{C})$ or $\mathrm{Spin}_{2n}(\mathbb{C})$ is nonzero if and only if $a_i'' \in 4\mathbb{N}$ for all $i$. 

    In such cases, write
    \begin{equation} \label{eq:4col}
    \mathcal{O} := \begin{cases} (4\alpha_{2q} \geq 4\alpha_{2q-1} > \dots > 4\alpha_2 \geq 4\alpha_1  > \alpha_0) &\mathrm{if}\   G = \mathrm{Sp}_{2n}(\mathbb{C}) \\
(4\alpha_{2q+1} > 4\alpha_{2q} \geq 4\alpha_{2q-1} > \dots > 4\alpha_2 \geq 4\alpha_1 > \alpha_0) &\mathrm{if}\  G = \mathrm{SO}_{2n}(\mathbb{C}) \end{cases}    
    \end{equation}
and 
\begin{equation} \label{eq:sigmaomega}
\begin{aligned}
\mathfrak{B}_{i}^+ &:= \begin{pmatrix}  \alpha_{2i}& \dots & 2 & 1\end{pmatrix}, \quad  &&\mathfrak{B}_{i}^-:= \begin{pmatrix} \alpha_{2i}-\frac{1}{2}& \dots& \frac{3}{2} & \frac{1}{2}\end{pmatrix} \\
\mathfrak{D}_{j}^+ &:= \begin{pmatrix} \alpha_{2j-1}-\frac{1}{2}& \dots& \frac{3}{2} & \frac{1}{2}\end{pmatrix}, \quad  &&\mathfrak{D}_{j}^- := \begin{pmatrix} \alpha_{2j-1}-1& \dots& 1 & 0\end{pmatrix}    
\end{aligned}
\end{equation}
Then, up to $\pm 1$, the character formula of the lift is equal to:

\bigskip
 \noindent $\bullet\ G = \mathrm{Sp}_{2n}(\mathbb{C})$:
        {\small\begin{align*}
        \mathrm{Lift} (R_{\emptyset}) = 2^q \sum_{\sigma_i, \omega_j}&\left(\mathrm{sgn}(\prod_{i}\sigma_i)\mathrm{sgn}(\prod_{j}\omega_j)\right) \\
       &\displaystyle   X\left(\begin{matrix}\quad \mathfrak{B}_{q}^+ \\ \sigma_{q}(\mathfrak{B}_{q}^-\end{matrix}  ;\ \begin{matrix} \mathfrak{B}_{q}^- \\ \mathfrak{B}_{q}^+)\end{matrix}\ ;\ \begin{matrix} \quad  \mathfrak{D}_{q}^+ \\ \omega_{q}(\mathfrak{D}_{q}^-\end{matrix}   ;  \begin{matrix} \mathfrak{D}_{q}^- \\\ \mathfrak{D}_{q}^+)\end{matrix} \ ; \cdots ;\  \begin{matrix}\quad \mathfrak{B}_{1}^+ \\ \sigma_{1}(\mathfrak{B}_{q}^-\end{matrix}  ;\ \begin{matrix} \mathfrak{B}_{1}^- \\ \mathfrak{B}_{1}^+)\end{matrix}\ ;\ \begin{matrix} \quad  \mathfrak{D}_{1}^+ \\ \omega_{1}(\mathfrak{D}_{1}^-\end{matrix}   ;  \begin{matrix} \mathfrak{D}_{1}^- \\\ \mathfrak{D}_{1}^+)\end{matrix} \ ;\begin{matrix}\quad \mathfrak{B}_{0}^+ \\ \sigma_{0}(\mathfrak{B}_{0}^-\end{matrix}  ; \ \begin{matrix} \mathfrak{B}_{0}^- \\ \mathfrak{B}_{0}^+)\end{matrix}\  \right);
        \end{align*}}
 \noindent $\bullet\ G = \mathrm{SO}_{2n}(\mathbb{C})$:
       {\small 
       \begin{align*}
           \mathrm{Lift}&(R_{\emptyset}) =\\ 
           &\displaystyle 2^q \sum_{\sigma_i, \omega_j} \left(\mathrm{sgn}(\prod_{i}\sigma_i)\mathrm{sgn}(\prod_{j}\omega_j)\right)  X\left(\begin{matrix} \quad \quad \mathfrak{D}_{q+1}^+ \\ \omega_{q+1}(\mathfrak{D}_{q+1}^-\end{matrix}  \ ; \ \begin{matrix} \mathfrak{D}_{q+1}^- \\\ \mathfrak{D}_{q+1}^+)\end{matrix} \ ; \ \begin{matrix}\quad \mathfrak{B}_{q}^+ \\ \sigma_{q}(\mathfrak{B}_{q}^-\end{matrix}  \ ; \ \begin{matrix} \mathfrak{B}_{q}^- \\ \mathfrak{B}_{q}^+) \end{matrix}\ ;\cdots ;\begin{matrix}\quad \mathfrak{D}_{1}^+ \\ \omega_{1}(\mathfrak{D}_{1}^-\end{matrix}  ; \ \begin{matrix} \mathfrak{D}_{1}^- \\ \mathfrak{D}_{1}^+)\end{matrix} ; \begin{matrix}\quad \mathfrak{B}_{0}^+ \\ \sigma_{0}(\mathfrak{B}_{0}^-\end{matrix}  ; \ \begin{matrix} \mathfrak{B}_{0}^- \\ \mathfrak{B}_{0}^+)\end{matrix}\  \right) \\
           +\ &\displaystyle 2^q \sum_{\sigma_i, \omega_j} \left(\mathrm{sgn}(\prod_{i}\sigma_i)\mathrm{sgn}(\prod_{j}\omega_j)\right)  \widehat{X}\left(\begin{matrix} \quad \quad \mathfrak{D}_{q+1}^+ \\ \omega_{q+1}(\mathfrak{D}_{q+1}^-\end{matrix}  \ ; \ \begin{matrix} \mathfrak{D}_{q+1}^- \\\ \mathfrak{D}_{q+1}^+)\end{matrix} \ ; \ \begin{matrix}\quad \mathfrak{B}_{q}^+ \\ \sigma_{q}(\mathfrak{B}_{q}^-\end{matrix}  \ ; \ \begin{matrix} \mathfrak{B}_{q}^- \\ \mathfrak{B}_{q}^+) \end{matrix}\ ;\cdots ;\begin{matrix}\quad \mathfrak{D}_{1}^+ \\ \omega_{1}(\mathfrak{D}_{1}^-\end{matrix}  ; \ \begin{matrix} \mathfrak{D}_{1}^- \\ \mathfrak{D}_{1}^+)\end{matrix} ; \begin{matrix}\quad \mathfrak{B}_{0}^+ \\ \sigma_{0}(\mathfrak{B}_{0}^-\end{matrix}  ; \ \begin{matrix} \mathfrak{B}_{0}^- \\ \mathfrak{B}_{0}^+)\end{matrix}\right)
       \end{align*}}
where $\sigma_i \in W(B_{\alpha_{2i}}) \times W(B_{\alpha_{2i}})$ and $\omega_j \in W(D_{\alpha_{2j-1}}) \times W(D_{\alpha_{2j-1}})$, and the 
$\widehat X\left(\cdots\right)$ in Type $D$ is obtained from $X\left(\cdots
\right)
$ by switching the $\mathfrak{D}_{q+1}^+$ expression in $X(\cdots)$ into:
$$\mathfrak{D}_{q+1}^+ = \left(\alpha_{2q+1}-\frac{1}{2}\ \dots \quad \frac{3}{2} \quad \frac{1}{2}\right)\quad  \mapsto \quad \widehat{\mathfrak{D}}_{q+1}^+ = \left(\alpha_{2q+1}-\frac{1}{2} \ \dots\quad \frac{3}{2} \quad \frac{-1}{2}\right)$$
\end{proposition}
\begin{proof}
By the Barbasch--Vogan character formula, $R_{\emptyset}$ is an
alternating sum of principal series representations $X(\lambda;w\lambda)$, where the coordinates of $\lambda$ decompose
into blocks
$$
(r,r-1,\ldots,1)
\qquad\text{and}\qquad
(r-1,r-2,\ldots,0)
$$
with $2r=a_i''$ by Remark \ref{r:infchar}. 
 Moreover, the Weyl group elements
appearing in this alternating sum come from a product of Weyl subgroups preserving
these blocks.

By Proposition \ref{prop:over2}, the lift of such a principal series is nonzero
only when
$\lambda-w\lambda\in(2\mathbb Z+1)^n.$

Since the Weyl groups act by signed permutations and signs do not affect
parity, this condition requires, in each block, that $w$ interchange the
even and odd coordinates.  Such a permutation exists if and only if the
block contains the same number of even and odd coordinates, which is
equivalent to $r$ being even.  Hence the lift can be nonzero only if
$a_i''\in4\mathbb Z$
for every $i$.

Suppose now that $a_i''=4\alpha_i$.  Separating the even and odd
coordinates in each block and dividing by $2$ gives precisely the
parameters $\mathfrak{B}_i^\pm$ and $\mathfrak{D}_i^\pm$ defined in \eqref{eq:sigmaomega}.
The Weyl group elements satisfying the parity condition are then
parametrized by
$$W(B_{\alpha_{2i}})\times W(B_{\alpha_{2i}})
\quad\text{and}\quad
W(D_{\alpha_{2j-1}})\times W(D_{\alpha_{2j-1}}).
$$
Keeping track of these two classes gives the factor $2^q$ in the stated
formulas. In type $D$, the remaining sign-parity distinction gives the
two terms in the formula, with the second obtained by replacing
$\mathfrak{D}_{q+1}^+$ by $\widehat{\mathfrak{D}}_{q+1}^+$.
Proposition \ref{prop:over2} then gives the stated character formulas.  
\end{proof}

The following example concerns the lift of trivial representations of $G = \mathrm{Sp}_{2n}(\mathbb{C})$ and $\mathrm{SO}_{2n}(\mathbb{C})$.
\begin{example} \label{eg:rho/2}
Let $G = \mathrm{Sp}_{2n}(\mathbb{C})$, then the character formula of the trivial representation is 
\[\mathrm{triv} =\sum_{w\in W(B_n)}\mathrm{sgn}(w)X\begin{pmatrix} n, n-1,\dots,1\\ w(n,n-1,\dots,1)\end{pmatrix}.\]
When $n$ is odd, the same parity argument as in Proposition \ref{lift_1_gln} implies that $\rho-w\rho\notin (2\mathbb{Z}+1)^n$ for any $w\in W(B_n)$. Hence, $\mathrm{Lift}(\mathrm{triv})=0$. Therefore 
$n = 2m$ is even as required in Proposition \ref{prop:lift} with $\mathcal{O} = (4m)$, and
$$\mathrm{Lift}(\mathrm{triv})= \mathrm{sgn}(y)\sum_{\sigma \in W(B_m) \times W(B_m)}\mathrm{sgn}(\sigma) X\begin{pmatrix} \frac{2m}{2},\dots,\frac{4}{2},\frac{2}{2} & \frac{2m-1}{2},\dots,\frac{3}{2},\frac{1}{2}\\ \sigma\Big((\frac{2m-1}{2},\dots,\frac{3}{2},\frac{1}{2}) & (\frac{2m}{2},\dots,\frac{4}{2},\frac{2}{2})\Big)\end{pmatrix},$$
where $y \in W(B_n) = W(B_{2m})$ with $y(0\ 1\ \dots\ 2m-1) := (0\ 2\ \dots\ 2m-2; 1\ 3\ \dots\ 2m-1)$ as in the proof of Proposition \ref{lift_1_gln} (from now on, we will omit the expression of $\mathrm{sgn}(y)$ in the formula of $\mathrm{Lift}$).

Let $G = \mathrm{SO}_{2n}(\mathbb{C})$. The character formula of the trivial representation is:
\begin{equation} \label{eq:dsign}
\begin{aligned}\mathrm{triv} &=\sum_{w\in W(D_n)}\mathrm{sgn}(w)X\begin{pmatrix} n-1,\dots,1,0\\ w(n-1,\dots,1,0)\end{pmatrix}\\ 
&= X\begin{pmatrix} n-1,\dots,1,0\\ n-1,\dots,1,0\end{pmatrix} - X\begin{pmatrix} n-1,\dots,1,0\\ n-1,\dots,0,1\end{pmatrix} - X\begin{pmatrix} n-1,\dots,1,0\\ n-1,\dots,0,-1\end{pmatrix} + \dots.
\end{aligned}
\end{equation}
By the same arguments as in the symplectic group case, the lift is nonzero if and only if $n = 2m$ is even with $\mathcal{O} = (4m,0)$, and
\begin{align*}
\mathrm{Lift}(\mathrm{triv})=\ &\sum_{\sigma \in W(D_m) \times W(D_m)}\mathrm{sgn}(\sigma) X\begin{pmatrix} \frac{2m-1}{2},\dots,\frac{3}{2},\frac{1}{2} & \frac{2m-2}{2},\dots,\frac{2}{2},\frac{0}{2}\\ \sigma\Big((\frac{2m-2}{2},\dots,\frac{2}{2},\frac{0}{2}) & (\frac{2m-1}{2},\dots,\frac{3}{2},\frac{1}{2})\Big)\end{pmatrix}\\ 
+
&\sum_{\sigma \in W(D_m) \times W(D_m)}\mathrm{sgn}(\sigma) X\begin{pmatrix} \frac{2m-1}{2},\dots,\frac{3}{2},\frac{1}{2} & \frac{2m-2}{2},\dots,\frac{2}{2},\frac{0}{2}\\ \sigma\Big((\frac{2m-2}{2},\dots,\frac{2}{2},\frac{0}{2}) & (\frac{2m-1}{2},\dots,\frac{3}{2},\frac{-1}{2})\Big)\end{pmatrix}
\end{align*}
Note that the two terms in the above formula have the same sign, which can be seen by the expansion of the character formula of trivial representation in Equation \eqref{eq:dsign}. By applying Proposition \ref{prop:BVprop} on the parameter $\begin{pmatrix} \frac{1}{2} & \frac{0}{2} \\ \frac{0}{2} & \frac{-1}{2} \end{pmatrix} \mapsto \begin{pmatrix} \frac{-1}{2} & \frac{0}{2} \\ \frac{0}{2} & \frac{1}{2} \end{pmatrix}$, the above expression can be rewritten as in the form in the Proposition.

Moreover, by \cite[Equation (4.6.1)]{BTs18}, the two summands of the above character formula are precisely the character formulas for the two $\rho/2$ genuine representations of $\mathrm{Spin}_{4m}(\mathbb{C})$ attached to the {\it model orbit} $\mathcal{O}' = (2m,2m-1,1)$, and they differ from each other by an outer automorphism. By denoting these representations as $\widetilde{\pi}(\mathcal{O}')$ and $\widetilde{\pi}(\mathcal{O}')^{\vee}$, one has
$$\mathrm{Lift}(\mathrm{triv})= \widetilde{\pi}(\mathcal{O}') \oplus \widetilde{\pi}(\mathcal{O}')^{\vee}.$$

As we will see in Theorem \ref{thm:main} below, this is a special case of a more general phenomenon when one computes the lift of the sum of special unipotent representations.
\end{example} 

Now we look at a case with multiple columns. This example will be used for the rest of the manuscript:
\begin{example} \label{eg:161284}
Let $G = \mathrm{Sp}_{40}(\mathbb{C})$ and $\mathcal{O} = (16,12,8,4,0)$. Then $(\alpha_4,\alpha_3,\alpha_2,\alpha_1,\alpha_0) = (4,3,2,1,0)$, and
$$R_{\emptyset} = \sum_{w \in W(B_8 \times D_6 \times B_4 \times D_2)} \mathrm{sgn}(w) X\left(\begin{matrix}
    \quad 87654321; & 543210; & 4321; & 10 \\
    w(87654321; & 543210; & 4321; & 10)
\end{matrix}\right)$$
Therefore,
$$\mathrm{Lift}(R_{\emptyset}) = \sum_{w \in W^{\flat}} \mathrm{sgn}(w) X\left(\begin{matrix}
    \quad \frac{8}{2}\frac{6}{2}\frac{4}{2}\frac{2}{2}; &\frac{7}{2}\frac{5}{2}\frac{3}{2}\frac{1}{2}; & \frac{5}{2}\frac{3}{2}\frac{1}{2}; &\frac{4}{2}\frac{2}{2}\frac{0}{2}; & \frac{4}{2}\frac{2}{2}; &\frac{3}{2}\frac{1}{2}; & \frac{1}{2}; &\frac{0}{2} \\
   w(\frac{7}{2}\frac{5}{2}\frac{3}{2}\frac{1}{2}; & \frac{8}{2}\frac{6}{2}\frac{4}{2}\frac{2}{2}; & \frac{4}{2}\frac{2}{2}\frac{0}{2}; & \frac{5}{2}\frac{3}{2}\frac{\pm 1}{2}; & \frac{3}{2}\frac{1}{2}; & \frac{4}{2}\frac{2}{2}; & \frac{0}{2}; & \frac{\pm 1}{2})
\end{matrix}\right)$$
where 
$$W^{\flat} = W(B_4 \times B_4 \times D_3 \times D_3 \times B_2 \times B_2 \times D_1 \times D_1),$$
and the $\pm$ sign in the formula comes from the same observation in Example \ref{eg:rho/2} for $G = \mathrm{SO}_{2n}(\mathbb{C})$. However, since the resulting virtual representation is in $\mathrm{Spin}_{2n+1}(\mathbb{C})$, one can switch all $\begin{pmatrix}
    0 \\ \frac{-1}{2}
\end{pmatrix} \mapsto \begin{pmatrix}
    0 \\ \frac{1}{2}
\end{pmatrix}$ directly, and get
$$\mathrm{Lift}(R_{\emptyset}) = 4\sum_{w \in W^{\flat}} \mathrm{sgn}(w) X\left(\begin{matrix}
    \quad \mathfrak{B}_2^+; &\mathfrak{B}_2^-; & \mathfrak{D}_2^+; &\mathfrak{D}_2^-; & \mathfrak{B}_1^+; &\mathfrak{B}_1^-; & \mathfrak{D}_1^+; &\mathfrak{D}_1^- \\
    w(\mathfrak{B}_2^-; &\mathfrak{B}_2^+; & \mathfrak{D}_2^-; &\mathfrak{D}_2^+; & \mathfrak{B}_1^-; &\mathfrak{B}_1^+; & \mathfrak{D}_1^-; &\mathfrak{D}_1^+)
\end{matrix}\right)$$
(recall the notations of $\mathfrak{B}^{\pm}$ and $\mathfrak{D}^{\pm}$ in Equation \eqref{eq:sigmaomega}, in particular $\mathfrak{B}_0^{\pm} = \emptyset$ since $\alpha_0 = 0$), where
\begin{equation} \label{eq-parity}
\begin{aligned}
\mathfrak{B}_2^+ \sqcup \mathfrak{D}_2^- \sqcup \mathfrak{B}_1^+ \sqcup \mathfrak{D}_1^- &= (4,3,2,1;\ 2,1,0;\ 2,1;\ 0)\\
\mathfrak{B}_2^- \sqcup \mathfrak{D}_2^+ \sqcup \mathfrak{B}_1^- \sqcup \mathfrak{D}_1^+ &= \left(\frac{7}{2},\frac{5}{2},\frac{3}{2},\frac{1}{2};\ \frac{5}{2}, \frac{3}{2}, \frac{1}{2};\ \frac{3}{2}, \frac{1}{2};\ \frac{1}{2}\right)
\end{aligned}
\end{equation}
after grouping together all the integral and half-integral terms.
\end{example}

\subsection{Main Theorem}
As hinted at the end of Example \ref{eg:rho/2}, one would anticipate that for $\mathcal{O} \subseteq \mathfrak{g}$  as given in Equation \eqref{eq:4col}, the sum of all special unipotent representations $R_{\emptyset}$ attached to $\mathcal{O}$ will be lifted to genuine unipotent representations of $\widetilde{G'} = \mathrm{Spin}_{2n+1}(\mathbb{C})$ and $\mathrm{Spin}_{2n}(\mathbb{C})$. 

To see which nilpotent orbit $\mathcal O' \subseteq \mathfrak g'$ these genuine unipotent representations of $\widetilde{G'}$ correspond to, note that the infinitesimal character of $\mathrm{Lift}(R_{\emptyset})$ is given by the $W$-conjugates of:
\begin{equation} \label{eq-h/4}
\frac{{}^{\vee}h}{4} = \bigsqcup_{i \geq 0} (\mathfrak{B}_{i}^+ \sqcup \mathfrak{B}_{i}^-) \sqcup \bigsqcup_{j \geq 1} (\mathfrak{D}_j^+ \sqcup \mathfrak{D}_j^-) \in (\mathfrak{h}')^*.
\end{equation}
By \cite{Br99} and \cite{B17}, or more recently \cite{MBM23}, this is precisely the infinitesimal character of the (non-special) unipotent representation attached to the nilpotent orbit $\mathcal{O}'$ whose column sizes are given by reordering the integers:
\begin{equation} \label{eq:oprime}
\bigsqcup_{i \geq 0}(2\alpha_{2i}+1,2\alpha_{2i}) \sqcup \bigsqcup_{j \geq 1}(2\alpha_{2j-1},2\alpha_{2j-1}-1)
\end{equation}
into descending order. In other words, all irreducible $(\mathfrak g', \widetilde{K}')$-modules having infinitesimal character $\frac{{}^{\vee}h}{4}$ must have associated variety greater than or equal to the closure of $\mathcal{O}'$.

\begin{example}
   Recall in Example \ref{eg:161284}, with $\mathcal{O} = (16,12,8,4,0)$ for $G = \mathrm{Sp}_{40}(\mathbb{C})$. Equation \eqref{eq:oprime} gives
    $$\mathcal{O}' = (9,8; 6,5; 5,4; 2,1; 1).$$
\end{example}

In Section \ref{sec:unique}, we will define genuine unipotent representations $\widetilde{\mathcal{U}}(\mathcal{O}')$ attached to the orbit $\mathcal{O}'$ and show the following uniqueness theorem for genuine unipotent representations:
\begin{theorem}[see Theorem \ref{thm:unique} below] \label{thm:preunique}
Let $\widetilde{G'} = \mathrm{Spin}_{2n+1}(\mathbb{C})$ or $\mathrm{Spin}_{2n}(\mathbb{C})$, and $\mathcal{O}' \subseteq \mathfrak{g}'$ be as given above. Then the irreducible representations in $\widetilde{\mathcal{U}}(\mathcal{O}')$ are precisely:
$$\begin{cases}  \widetilde{\pi}(\mathcal{O}') := J\left(\begin{matrix}\mathfrak{B}_{q}^+ \\ \mathfrak{B}_{q}^-\end{matrix}  \ ; \ \begin{matrix} \mathfrak{B}_{q}^- \\ \mathfrak{B}_{q}^+\end{matrix}\ ;\ \begin{matrix}  \mathfrak{D}_{q}^+ \\ \mathfrak{D}_{q}^-\end{matrix}   ; \begin{matrix} \mathfrak{D}_{q}^- \\\ \mathfrak{D}_{q}^+\end{matrix} \ ;\ \cdots \ ;\  \begin{matrix}\mathfrak{B}_{0}^+ \\ \mathfrak{B}_{0}^-\end{matrix}  \ ; \ \begin{matrix} \mathfrak{B}_{0}^- \\ \mathfrak{B}_{0}^+\end{matrix}\right)  &\textrm{if}\ \widetilde{G'} = \mathrm{Spin}_{2n+1}(\mathbb{C});
 \\
 \widetilde{\pi}(\mathcal{O}') := J\left(\begin{matrix}  \mathfrak{D}_{q+1}^+ \\ \mathfrak{D}_{q+1}^-\end{matrix}   ; \begin{matrix} \mathfrak{D}_{q+1}^- \\\ \mathfrak{D}_{q+1}^+\end{matrix} \ ;\ \begin{matrix}\mathfrak{B}_{q}^+ \\ \mathfrak{B}_{q}^-\end{matrix}  \ ; \ \begin{matrix} \mathfrak{B}_{q}^- \\ \mathfrak{B}_{q}^+\end{matrix}\ ;\ \cdots \ ;\  \begin{matrix}\mathfrak{B}_{0}^+ \\ \mathfrak{B}_{0}^-\end{matrix}  \ ; \ \begin{matrix} \mathfrak{B}_{0}^- \\ \mathfrak{B}_{0}^+\end{matrix}\right)\ \textrm{and}\ \widetilde{\pi}(\mathcal{O}')^{\vee}  & \textrm{if}\ \widetilde{G'} = \mathrm{Spin}_{2n}(\mathbb{C})
\end{cases}$$
where the Zhelobenko parameter of $\widetilde{\pi}(\mathcal{O}')^{\vee}$ can be obtained from that of $\widetilde{\pi}(\mathcal{O}')$ by switching the first  $\mathfrak{D}_{q+1}^+$'s into $\widehat{\mathfrak{D}}_{q+1}^+$ (see Theorem \ref{thm:unique} below).
In other words, there is a {\bf unique} genuine unipotent representation attached to $\mathcal{O}'$ up to outer automorphism of $\mathfrak{g}'$. 
\end{theorem}
These representations are known to be unitary by \cite{Br99} and \cite{B17}. Furthermore, it was shown in \cite{WZ23} by the second and third-named authors that they are {\it precisely} the `building blocks' of the genuine unitary dual of $\widetilde{G'}$.

Using the uniqueness theorem, we will prove in Section \ref{sec:main} that:
\begin{theorem} \label{thm:main}
    Let $\mathcal{O} \subseteq \mathfrak{g}$ be a nilpotent orbit given in Equation \eqref{eq:4col}. Then the virtual representation $\mathrm{Lift}(R_{\emptyset})$, whose character formula is given  in Proposition \ref{prop:lift},  decomposes as follows:
    $$\mathrm{Lift}(R_{\emptyset}) = \mathrm{Lift}\left(\bigoplus_{\sigma \in \overline{A}(\mathcal{O})^{\wedge}} X_{\sigma}\right) = \begin{cases} \widetilde{\pi}(\mathcal{O}')^{\oplus  |\overline{A}(\mathcal{O})^{\wedge}|} & \mathrm{if}\ G = \mathrm{Sp}_{2n}(\mathbb{C})\\
     \widetilde{\pi}(\mathcal{O}')^{\oplus  |\overline{A}(\mathcal{O})^{\wedge}|} \oplus (\widetilde{\pi}(\mathcal{O}')^{\vee})^{\oplus  |\overline{A}(\mathcal{O})^{\wedge}|} &\mathrm{if}\ G = \mathrm{SO}_{2n}(\mathbb{C}) \end{cases}$$
where $|\overline{A}(\mathcal{O})^{\wedge}| =2^q$ in both cases.
\end{theorem}

Given the above theorem, we obtain the lift of each special unipotent representation $X_\sigma$:


\begin{corollary}
    Let $\mathcal{O} \subseteq \mathfrak{g}$ be a nilpotent orbit given in Equation \eqref{eq:4col}. Then the lift of each special unipotent representation $X_{\sigma} \in \mathcal{U}(\mathcal{O})$ is 
    $$\mathrm{Lift}(X_{\sigma})  = \begin{cases}
        \widetilde{\pi}(\mathcal{O}')  &\mathrm{if}\ G = \mathrm{Sp}_{2n}(\mathbb{C}) \\
        \widetilde{\pi}(\mathcal{O}') \oplus  \widetilde{\pi}(\mathcal{O}')^{\vee} &\mathrm{if}\ G = \mathrm{SO}_{2n}(\mathbb{C})
    \end{cases}.$$
    In particular, the lift of special unipotent representations $X_{\sigma}$ remains unitary.
\end{corollary}

\begin{proof}
Since
$\overline A(\mathcal O)\simeq(\mathbb Z/2\mathbb Z)^q$, 
under the identification of its conjugacy
classes with subsets $I\subseteq\{1,\ldots,q\}$,
Theorem \ref{thm:BV}(b)
gives
$$X_\sigma=\frac{1}{|\overline A(\mathcal O)|}
\sum_{I\in\overline A(\mathcal O)}
\mathrm{tr}_\sigma(I)R_I.
$$
We claim that $\mathrm{Lift}(R_I)=0$ for $I\neq \emptyset$. 

Indeed, for $I\neq\emptyset$, choose $i\in I$.  In the character formula for
$R_I$, the corresponding pair of blocks is changed from
$(r_{2i},\ldots,1),\  (r_{2i-1}-1,\ldots,0)$
(which are the relevant blocks for $R_\emptyset$)
to
$(r_{2i},\ldots,1,0),\  (r_{2i-1}-1,\ldots,1).$ Since $r_j$ is even for all $j$, each of these modified blocks has an
unequal number of even and odd coordinates.  Hence no Weyl group element
preserving the blocks can interchange even and odd coordinates. Proposition \ref{prop:over2} therefore implies that $\mathrm{Lift}(R_I)=0.$ 

Since $\mathrm{tr}_\sigma(e)=1$, it follows that
$$
\mathrm{Lift}(X_\sigma)=
\frac{1}{|\overline A(\mathcal O)|} \mathrm{Lift}(R_{\emptyset}),
$$
and the result follows from Theorem
\ref{thm:main}.
\end{proof}

\begin{remark} \label{rmk-stable}
    It is more natural to study the lift of the sum $\mathrm{Lift}(R_I)$ rather than the individual $\mathrm{Lift}(X_{\sigma})$'s. Indeed,  
    for each $\mathcal{O}_I \in \mathrm{SP}(\mathcal O)$ (cf. Remark \ref{rmk-stablesum}), the $R_{I}$'s defined in Theorem \ref{thm:BV}(c) are the (stable) virtual sums of special unipotent representations attached to the Weyl group representation $\mathrm{Spr}(\mathcal{O}_I)$ under coherent continuation representation. 
    
    In a forthcoming work \cite{TW}, we will 
    study the lift of stable sums of representations for all real, simply laced linear reductive groups under the set-up of Adams--Herb \cite{AH10}. In such a case, $R_{\emptyset}$ is the {\bf only} stable sum (up to scalar multiples) whose lift is guaranteed to be a sum of unipotent representations. As for the other $R_I$'s, the lifts may no longer be unipotent.
    
    As an example, for the spherical unipotent representation $\displaystyle X_{\mathrm{triv}} = \frac{1}{|\mathcal{P}(q)|}\sum_{I \in \mathcal{P}(q)} R_I$ attached to $\mathcal{O}$, one has that:
    \begin{center}
    $\mathrm{Lift}(X_{\mathrm{triv}}) \neq 0$\ \ $\Leftrightarrow$\  \ 
    $\mathcal{O} = \begin{cases} (\gamma_{2q}, \dots, \gamma_{1}, \gamma_{0}) &\mathrm{if}\ G = \mathrm{Sp}_{2n}(\mathbb{C})\\ 
    (\gamma_{2q+1}, \gamma_{2q}, \dots, \gamma_{1},\gamma_{0}) &\mathrm{if}\ G = \mathrm{SO}_{2n} (\mathbb{C})\end{cases},$
    where $\gamma_{2i} \equiv \gamma_{2i-1} \equiv 0, 2 (\mathrm{mod}\ 4)$. 
   \end{center}
    For instance, if one takes $\mathcal{O} = (8,6,2)$ in $\mathrm{SO}_{16}(\mathbb{C})$ (where Equation \eqref{eq:4col} is not satisfied and hence $\mathrm{Lift}(R_{\emptyset}) =  0$), then $\mathrm{Lift}(X_{\mathrm{triv}}) \neq 0$ and is {\bf not} unitarizable. Namely, it has an irreducible factor labeled by $\begin{pmatrix} 2 & 2 \\ 0 & 0 \end{pmatrix}$ on the table at the end of \cite[Section 3]{WZ23}. 
\end{remark}

We end this section by relating $\mathcal{O}'$ and a specific Weyl group representation which occurs in the character formula of $\mathrm{Lift}(R_{\emptyset})$ (cf. Example \ref{eg:161284}). 
\begin{proposition} \label{prop-whw'}
    Let $\mathcal{O}$ and $\mathcal{O}'$ be nilpotent orbits in $\mathfrak{g}$ and $\mathfrak{g}'$ given by Equations \eqref{eq:4col} and \eqref{eq:oprime} respectively.
    Since $\mathcal O$ is as in Equation \eqref{eq:4col}, we have
$2n=\sum_i 4\alpha_i$, and hence $n$ is even.
    Consider 
    $$W_{\frac12}^{\flat} :=  W(\prod_{i \geq 1} B_{\alpha_{2i}} \times \prod_{j \geq 1} D_{\alpha_{2j-1}})  \leq  W_{\frac12} := \begin{cases} W(B_{n/2}) & \text{if}\ \widetilde{G'} = \mathrm{Spin}_{2n+1}(\mathbb{C})\\
W(D_{n/2}) & \text{if}\ \widetilde{G'} = \mathrm{Spin}_{2n}(\mathbb{C}) \end{cases},$$
    and 
    $$W^{\flat} := W_{\frac12}^{\flat} \times W_{\frac12}^{\flat},$$ 
    so that $W^{\flat} \leq W$ is the Weyl subgroup appearing in the character formula of $\mathrm{Lift}(R_{\emptyset})$ in Proposition \ref{prop:lift} (cf. Example \ref{eg:161284}). Then one has:
    $$j_{W^{\flat}}^{W}(\mathrm{sgn}) = \mathrm{Spr}(\mathcal{O}').$$
\end{proposition}
\begin{proof}
By the standard calculation of truncated induction for classical Weyl groups, one has
    $$j_{W^{\flat}}^{W}(\mathrm{sgn}) = (\cdots,\alpha_{2j-1}, \alpha_{2j-1}, \cdots, \alpha_1, \alpha_1) \times(\cdots, \alpha_{2i}, \alpha_{2i}, \cdots, \alpha_2,\alpha_2, \alpha_0,\alpha_0).$$ 
   In particular, $j_{W^{\flat}}^{W}(\mathrm{sgn})$ is irreducible.  
    
    Under the Springer correspondence, the above Weyl group representation corresponds precisely to the nilpotent orbit $\mathcal O'$ whose column partition is given in Equation \eqref{eq:oprime}.
 Thus, the result follows.  
\end{proof}

\section{Proof of Main Theorem} 
\subsection{Genuine Unipotent Representations - A Uniqueness Theorem} \label{sec:unique} In this section, we assume $\widetilde{G'} = \mathrm{Spin}_{2n+1}(\mathbb{C})$ or $\mathrm{Spin}_{2n}(\mathbb{C})$, and $\mathcal{O}' \subseteq \mathfrak{g}' = \mathfrak{so}_{2n+1}(\mathbb{C})$ or $\mathfrak{so}_{2n}(\mathbb{C})$ is a nilpotent orbit given by Equation \eqref{eq:oprime}. By our discussions in the previous section, the lift of special unipotent representations has infinitesimal character equal to 
$\frac{{}^{\vee}h}{4}$ as given in Equation \eqref{eq-h/4}. Note that for any $\beta \in \mathfrak{B}_{i}^{\pm}$ and $\delta \in \mathfrak{D}_j^{\mp}$, one always has
$\beta \equiv \delta (\mathrm{mod}\ \mathbb{N})$ (cf. Equation \eqref{eq-parity}). Consider the block consisting of all genuine representations with infinitesimal character $\frac{{}^{\vee}h}{4}$:
$$\mathcal{B} =  \mathrm{Span}\left\{\
    J(w,w') := J\left(\begin{matrix}\bigsqcup_{i \geq 0} \mathfrak{B}_{i}^+ \sqcup \bigsqcup_{j \geq 1} \mathfrak{D}_j^- \\ w(\bigsqcup_{i \geq 0} \mathfrak{B}_{i}^- \sqcup \bigsqcup_{j \geq 1}\mathfrak{D}_j^+)\end{matrix}  ;\ \begin{matrix}\bigsqcup_{i \geq 0} \mathfrak{B}_{i}^- \sqcup \bigsqcup_{j \geq 1} \mathfrak{D}_j^+ \\ w'(\bigsqcup_{i \geq 0} \mathfrak{B}_{i}^+ \sqcup \bigsqcup_{j \geq 1}\mathfrak{D}_j^-)\end{matrix}  \right) \ \Bigg|\ w, w' \in W_{\frac12} \right\}$$
for $\widetilde{G'} = \mathrm{Spin}_{2n+1}(\mathbb{C})$, and
\begin{align*}
\mathcal{B} =  \mathrm{Span}\left\{ J(w,w')  \ |\ w, w' \in W_{\frac12} \right\}, \quad \mathcal{B}^{\vee} = \mathrm{Span}\left\{J(w,w')^{\vee}\ |\ w, w' \in W_{\frac12}\right\}
\end{align*}
for $\widetilde{G'} = \mathrm{Spin}_{2n}(\mathbb{C})$ (here $^{\vee}$ is as defined in Example \ref{eg:rho/2}). Coherent continuation gives a $(W_{\frac12} \times W_{\frac12})$-module structure of these blocks.

\medskip
As discussed in the previous section, we make the following:
\begin{definition}
Let $\mathcal{O}' \subseteq \mathfrak{g}'$ be the nilpotent orbit given in Equation \eqref{eq:oprime}. The {\bf genuine unipotent representations} attached to $\mathcal{O}'$ are defined by:
    \begin{align*}
    \widetilde{\mathcal{U}}(\mathcal{O}') = \{J(w,w')\ |\ \mathcal{V}(\mathrm{LAnn}(J(w,w'))) = \mathcal{V}(\mathrm{RAnn}(J(w,w')))  = \overline{\mathcal{O}'} \}
    \end{align*}
(cf. Theorem \ref{thm:BV}(d)). Consequently, the wavefront set of each $J(w,w') \in  \widetilde{\mathcal{U}}(\mathcal{O}')$ is equal to $\overline{\mathcal{O}'} = \mathcal{V}(I(\lambda_{\mathcal{O}'}))$, and all other genuine representations have associated variety strictly bigger than $\overline{\mathcal{O}'}$ in the closure ordering of nilpotent orbits.
\end{definition}

Note that 
$\widetilde{\pi}(\mathcal{O}') = J(e,e)$
using the notations in Theorem \ref{thm:preunique}. Also, $\widetilde{\pi}(\mathcal{O}'), \widetilde{\pi}(\mathcal{O}')^{\vee} \in \widetilde{\mathcal{U}}(\mathcal{O}')$ are unitarizable representations by the main results of \cite{Br99} for $\mathrm{Spin}_{2n}(\mathbb{C})$ and \cite[Chapter 7]{B17} for $\mathrm{Spin}_{2n+1}(\mathbb{C})$. 
The main theorem in this section is, these representations exhausts all possibilities of $\widetilde{\mathcal U}(\mathcal O')$:
\begin{theorem} \label{thm:unique}
Let $\widetilde{G'}=\mathrm{Spin}_{2n+1}(\mathbb{C})$ or $\mathrm{Spin}_{2n}(\mathbb{C})$, and let $\mathcal{B}$ be  the block of genuine representations of $\widetilde{G'}$ as given above. Among all genuine irreducible representations $J(w,w')$ in $\mathcal{B}$, $J(e,e)$ is the only unipotent representation. In the case 
$\widetilde{G'}=\mathrm{Spin}_{2n}(\mathbb C)$,  the corresponding statement  for $\mathcal B^\vee$ follows by outer automorphism. 
Consequently,
    $$\widetilde{\mathcal{U}}(\mathcal{O}') = \begin{cases} \{\widetilde{\pi}(\mathcal{O}')\}
&\textrm{for}\ \widetilde{G'} = \mathrm{Spin}_{2n+1}(\mathbb{C}) \\ 
\{\widetilde{\pi}(\mathcal{O}'), \widetilde{\pi}(\mathcal{O}')^{\vee}\} &\textrm{for}\ \widetilde{G'} = \mathrm{Spin}_{2n}(\mathbb{C}) \end{cases}.$$
\end{theorem}

Rather than studying the $W_{\frac12} \times W_{\frac12}$-module $\mathcal{B}$ directly, let
$$
    \widetilde{G}_{\frac12}' = \begin{cases} \mathrm{Spin}_{n+1}(\mathbb{C}) &\textrm{for}\  \widetilde{G'} = \mathrm{Spin}_{2n+1}(\mathbb{C}) \\
    \mathrm{Spin}_{n}(\mathbb{C}) &\textrm{for}\  \widetilde{G'} = \mathrm{Spin}_{2n}(\mathbb{C})\end{cases}
$$
and consider the `half-blocks' of $\mathcal{B}$ for $\widetilde{G}_{\frac12}'$:
$$\mathcal{B}_{\frac12} :=  \mathrm{Span}\left\{\
      J(w):= J\left(\begin{matrix}\bigsqcup_{i \geq 0} \mathfrak{B}_{i}^+ \sqcup \bigsqcup_{j \geq 1} \mathfrak{D}_j^- \\ w(\bigsqcup_{i \geq 0} \mathfrak{B}_{i}^- \sqcup \bigsqcup_{j \geq 1}\mathfrak{D}_j^+)\end{matrix}\right) \ \Bigg|\ w \in W_{\frac12} \right\},$$
$$\mathcal{B}_{\frac12}' :=  \mathrm{Span}\left\{\
       J(w)' := J\left(\begin{matrix}\bigsqcup_{i \geq 0} \mathfrak{B}_{i}^- \sqcup \bigsqcup_{j \geq 1} \mathfrak{D}_j^+ \\ w'(\bigsqcup_{i \geq 0} \mathfrak{B}_{i}^+ \sqcup \bigsqcup_{j \geq 1}\mathfrak{D}_j^-)\end{matrix}\right) \ \Bigg|\ w' \in W_{\frac12} \right\}$$
One would like to understand the $W_{\frac12}$-module structure of the coherent continuation representation in $\mathcal{B}_{\frac12}$ and $\mathcal{B}_{\frac12}'$. Note that the irreducibles in $\mathcal{B}_{\frac12}'$ are just the contragredients of $\mathcal{B}_{\frac12}$, so one may further reduce to studying the block $\mathcal{B}_{\frac12}$. 

\begin{lemma} \label{lem-P}
Let $\lambda_L$, $\lambda_R$ be the dominant forms of $\bigsqcup_{i \geq 0} \mathfrak{B}_{i}^+ \sqcup \bigsqcup_{j \geq 1} \mathfrak{D}_j^-$ and $\bigsqcup_{i \geq 0} \mathfrak{B}_{i}^- \sqcup \bigsqcup_{j \geq 1}\mathfrak{D}_j^+$ respectively. Then the annihilator ideals $I(\lambda_L), I(\lambda_R)$ of $U(\mathfrak{g})$ of the irreducible Verma modules with highest weights $\lambda_L-\rho$ and $\lambda_R-\rho$ respectively satisfy:
$$\mathcal{V}(I(\lambda_L)) = \mathcal{V}(I(\lambda_R)) = \overline{\mathcal{P}}$$ 
where
\begin{equation} \label{eq-Pdef}
\mathcal{P} := 
\begin{cases}
(2\alpha_{2q}+1, 2\alpha_{2q-1}-1,   \dots, 2\alpha_2+1, 2\alpha_1-1, 2\alpha_0+1) &\text{if } \widetilde{G}_{\frac12}' = \mathrm{Spin}_{n+1}(\mathbb{C}); \\
(2\alpha_{2q+1}, 2\alpha_{2q}, \dots, 2\alpha_0)&\text{if } \widetilde{G}_{\frac12}' = \mathrm{Spin}_{n}(\mathbb{C})
\end{cases}
\end{equation}
\end{lemma}
\begin{proof}
The integral Weyl group of $\lambda_L$ and $\lambda_R$ are $W_{\lambda_L} = W_{\lambda_R} = W_{\frac12}$, and the reflections in $W_{\frac12}$ stabilizing $\lambda_L$ and $\lambda_R$ forms the Levi subgroups $W_{\frac12}^{\lambda_L}$ and $W_{\frac12}^{\lambda_R}$ respectively.

Let $w_L \in W_{\frac12}^{\lambda_L}$ and $w_R \in W_{\frac12}^{\lambda_R}$ be the longest elements in the stabilizer subgroups, and $V^L(w_L)$, $V^L(w_R)$ be the left cell representations of $W_{\frac12}$, each containing a unique special representation of multiplicity one:
$$\sigma_L \leq V^L(w_L), \quad \quad \sigma_R \leq V^L(w_R)$$
Then one of the main results in \cite{BV82} and \cite{BV83} implies that 
$\mathcal{V}(I(\lambda_{\bullet}))  = \overline{\mathcal{P}_{\bullet}}$ such that $\mathrm{Spr}(\mathcal{P}_{\bullet}) = \sigma_{\bullet}$ for $\bullet = L$ or $R$. Therefore, it suffices to show that
$$\sigma_L = \sigma_R = \mathrm{Spr}(\mathcal{P}).$$
This can be proved directly by noting that (cf. \cite[Proposition 3.15]{BV83})
$$V^L(w_L) = j_{W_{\frac12}^{\lambda_L}}^{W_{\frac12}}(\mathrm{triv}), \quad \quad V^L(w_R) =j_{W_{\frac12}^{\lambda_R}}^{W_{\frac12}}(\mathrm{triv})$$
(see Proposition \ref{prop:uniquehalf} below for some explicit left cell computations). 
Alternatively, consider the BVLS dual of $\mathcal{P}_{\bullet}$ (Equation \eqref{eq-BVLS}):
$$\mathcal{P}_{\bullet} \longleftrightarrow {}^{\vee}\mathcal{P}_{\bullet}:=\mathrm{Ind}_{{}^{\vee}\mathfrak{m}_{\bullet}}^{{}^{\vee}\mathfrak{g}_{\frac12}'}(0),$$
which is a Richardson orbit whose Levi subalgebra ${}^{\vee}\mathfrak{m}_{\bullet}$ corresponds to the singular roots of $\lambda_{\bullet}$ in ${}^{\vee}\mathfrak{g}_{\frac12}'$. Then one has 
$$\mathrm{Spr}(\mathcal{P}_{\bullet}) \otimes \mathrm{sgn} = \mathrm{Spr}({}^{\vee}\mathcal{P}_{\bullet}).$$ 
One can directly check that ${}^{\vee}\mathcal{P}_{L} = {}^{\vee}\mathcal{P}_{R}$, and its dual is equal to $\mathcal{P}$ as given in the lemma.
\end{proof}

\begin{example} \label{eg-P}
We continue with Example \ref{eg:161284}, where $\widetilde{G}_{\frac12}' = \mathrm{Spin}_{21}(\mathbb{C})$ and 
\begin{align*}
\mathfrak{B}_2^+ \sqcup \mathfrak{D}_2^- \sqcup \mathfrak{B}_1^+ \sqcup \mathfrak{D}_1^- &= (4,3,2,1;\ 2,1,0;\ 2,1;\ 0) \ \longrightarrow \ \lambda_L = (4;\ 3;\ 2,2,2;\ 1,1,1;\ 0,0)\\
\mathfrak{B}_2^- \sqcup \mathfrak{D}_2^+ \sqcup \mathfrak{B}_1^- \sqcup \mathfrak{D}_1^+ &= \left(\frac{7}{2},\frac{5}{2},\frac{3}{2},\frac{1}{2};\ \frac{5}{2}, \frac{3}{2}, \frac{1}{2};\ \frac{3}{2}, \frac{1}{2};\ \frac{1}{2}\right) \ \longrightarrow \ \lambda_R = \left(\frac{7}{2};\ \frac{5}{2},\frac{5}{2};\ \frac{3}{2}, \frac{3}{2}, \frac{3}{2};\ \frac{1}{2}, \frac{1}{2},  \frac{1}{2}, \frac{1}{2}\right).
\end{align*} 
Then $\mathfrak{m}_L$ (and $W_{\frac12}^{\lambda_L}$) correspond to the subroot system of type $A_0 \times A_0 \times A_2 \times A_2 \times B_2$, while
$\mathfrak{m}_R$ (and $W_{\frac12}^{\lambda_R}$) correspond to the subroot system of type $A_0 \times A_1 \times A_2 \times A_3$. 
In both cases, Richardson orbits in ${}^{\vee}\mathfrak{g}_{\frac12}' = \mathfrak{sp}_{20}(\mathbb{C})$ are equal to 
\begin{align*}
{}^{\vee}\mathcal{P}_{L} &= \mathrm{Ind}_{{}^{\vee}\mathfrak{m}_{L}}^{{}^{\vee}\mathfrak{g}_{\frac12}'}(0) = (4,3,3,3,3,1,1,1,1)_C \\
{}^{\vee}\mathcal{P}_{R} &= \mathrm{Ind}_{{}^{\vee}\mathfrak{m}_{R}}^{{}^{\vee}\mathfrak{g}_{\frac12}'}(0) =(4,4,3,3,2,2,1,1)_C
\end{align*}
where $\pi_C$ is the $C$-collapse of the partition $\pi$ so that it defines a nilpotent orbit of type $C$. In both cases, ${}^{\vee}\mathcal{P} = (4,4,3,3,2,2,1,1)$ and hence
$\mathcal{P} = \mathrm{Spr}^{-1}(\mathrm{Spr}({}^{\vee}\mathcal{P}) \otimes \mathrm{sgn}) =(9,5,5,1,1)$ as stated in the Lemma.
\end{example}

\begin{remark} \label{rmk:wh}
    Recall the Weyl subgroup $W_{\frac12}^{\flat}$ of $W_{\frac12}$ given in Proposition \ref{prop-whw'}. One can check by direct computation as in Proposition \ref{prop-whw'} that the nilpotent orbit $\mathcal{P}$ satisfies
    \begin{center}
    $\mathrm{Spr}(\mathcal{P}) = j_{W_{\frac12}^{\flat}}^{W_{\frac12}}(\mathrm{sgn})$. 
    \end{center}
Indeed, the truncated induced module $j_{W_{\frac12}^{\flat}}^{W_{\frac12}}(\mathrm{sgn})$ is irreducible with
$$j_{W_{\frac12}^{\flat}}^{W_{\frac12}}(\mathrm{sgn}) = (\cdots, \alpha_{2j-1}, \cdots, \alpha_3,\alpha_1) \times(\cdots, \alpha_{2j}, \cdots, \alpha_4, \alpha_2,\alpha_0).$$
\end{remark}

Now we go back to studying $\mathcal{B}_{\frac12}$ -- among all irreducible representations in $\mathcal{B}_{\frac12}$, let 
    $$\widetilde{\mathcal{U}}(\mathcal{P}) := \{J(w)\ |\  \mathcal{V}\mathrm{(LAnn}(J(w))) = \mathcal{V}\mathrm{(RAnn}(J(w))) = \overline{\mathcal{P}}\}$$
be the collection of irreducible representations with the smallest possible associated variety $\overline{\mathcal{P}}$.

\begin{proposition} \label{prop:uniquehalf}
Retaining the above settings, one has:
    $$\widetilde{\mathcal{U}}(\mathcal{P}) = \{J(e)\}.$$
\end{proposition}
\begin{proof}
By Proposition \ref{prop:BVprop}, one can check that $J(e) \cong J(\lambda_L; \lambda_R)$. Therefore, it follows from Lemma \ref{lem-P} that:
$$\mathcal{V}\mathrm{(LAnn}(J(e)))  = \mathcal{V}\mathrm{(RAnn}(J(e))) =  \mathcal{V}(I (\lambda_L))= \mathcal{V}(I (\lambda_R)) = \overline{\mathcal P},$$ 
i.e. $J(e) \in \widetilde{\mathcal U}(\mathcal P)$. Therefore, we are left to prove that $J(e)$ is the only representation in $\widetilde{\mathcal{U}}(\mathcal{P})$. 

By \cite[Corollary 5.24]{BV85}, the cardinality of $\widetilde{\mathcal{U}}(\mathcal{P})$ is equal to the number of irreducible representations in 
    \begin{equation} \label{eq-LRcell}
    V^L(w_L) \cap V^L(w_R),
    \end{equation}
    where $w_L, w_R \in W_{\frac12}$ are as given in Lemma \ref{lem-P} (note that $w_R$ is an involution, so that $V^L(w_R) = V^R(w_R)$). By the lemma, one already has $\mathrm{Spr}(\mathcal{P}) \in V^L(w_L) \cap V^L(w_R)$, and we are left to prove that this is the only irreducible representation in common.
    
    Begin with $V^L(w_L)$. Recall that $w_L$ is the longest element in the stabilizer subgroup of $\lambda_L$ consisting of integer coordinates. Let $\mathcal{Q}$ be a nilpotent orbit
    $$\mathcal{Q} = \begin{cases} (2\alpha_{2q}, 2\alpha_{2q-1},\dots, 2\alpha_1,2\alpha_0) \subseteq \mathfrak{sp}_n & \textrm{if}\ \widetilde{G}_{\frac12}' = \mathrm{Spin}_{n+1}(\mathbb{C})\\
    (2\alpha_{2q+1}, 2\alpha_{2q}, 2\alpha_{2q-1},\dots, 2\alpha_1,2\alpha_0) \subseteq \mathfrak{so}_n & \textrm{if}\ \widetilde{G}_{\frac12}' = \mathrm{Spin}_{n}(\mathbb{C}) \end{cases}.$$
    By \cite[Proposition 5.28]{BV85}, $V^L(w_L)$
    consists of Springer representations:
    $$V^L(w_L) = \{\mathrm{Spr}
    (\mathcal{Q}_{I})\ |\ I \subseteq \{1,2,\dots,q\}\},$$
    where $\mathcal{Q}_I$ is defined in Proposition \ref{prop:barao}. More precisely, $\mathrm{Spr}(Q_{\emptyset}) = \mathrm{Spr}(\mathcal{P}) = (\cdots, \alpha_{2j-1}, \cdots, \alpha_3,\alpha_1) \times(\cdots, \alpha_{2j}, \cdots, \alpha_4,\alpha_2,\alpha_0)$ is the unique special $W_{\frac12}$-representation in $V^L(w_L)$, and 
    $$\mathrm{Spr}(\mathcal{Q}_I) = (\dots, \alpha_{2j-1}', \dots, \alpha_3', \alpha_1') \times (\dots, \alpha_{2j}', \dots, \alpha_4', \alpha_2',\alpha_0),$$
    where $(\alpha_{2j-1}',\alpha_{2j}')$ is given by $$\begin{cases} (\alpha_{2j-1},\alpha_{2j}) & \textrm{if}\ j \notin I \\
    (\alpha_{2j-1} + \min \{\alpha_{2j+1} - \alpha_{2j-1}, \alpha_{2j} - \alpha_{2j-2}\}, \alpha_{2j} - \min \{\alpha_{2j+1} - \alpha_{2j-1}, \alpha_{2j} - \alpha_{2j-2}\}) & \textrm{if}\ j \in I \end{cases}.$$

    As for $V^L(w_R)$, note that it is isomorphic to $j_{W_{\frac12}^{\lambda_R}}^{W_{\frac12}}(\mathrm{triv})$ by the proof of Lemma \ref{lem-P}. Here $W_{\frac12}^{\lambda_R}$ is of type $(\cdots \times A_{\beta_2-1} \times A_{\beta_1-1})$, where $(\cdots, \beta_2, \beta_1) = (\cdots, \alpha_2, \alpha_1,\alpha_0)^T$ is the dual partition of the $\alpha$'s, so that one can compute the (non-truncated) induction $$\mathrm{Ind}_{W_{\frac12}^{\lambda_R}}^{W_{\frac12}}(\mathrm{triv}) = \mathrm{Ind}_{\prod_k W(A_{\beta_k-1})}^{W_{\frac12}}(\mathrm{triv})$$
    using the Littlewood-Richardson rule (see Example below), and check that none of the irreducible representations in the induced module contains $\mathrm{Spr}
    (\mathcal{Q}_{I})$ for $I \neq \emptyset$. Then the result follows.
\end{proof}

\begin{example}
In Example \ref{eg-P}, one has $\mathcal{Q} = (8,6,4,2,0)$ in  $\mathfrak{sp}_{20}(\mathbb{C})$, then
\begin{align*}V^L(w_L) &= \{\mathrm{Spr}((8,6,4,2,0)),\ \mathrm{Spr}((9,5,4,2,0)),\ \mathrm{Spr}((8,6,5,1,0)),\ \mathrm{Spr}((9,5,5,1,0))\} \\
&= \{31 \times 42, 51 \times 22, 33 \times 4, 53 \times 2\}
\end{align*}
as given in Example \ref{eg:charform}. On the other hand, one has 
$W_{\frac12}^{\lambda_R} = W(A_3 \times A_2 \times A_1 \times A_0)$ and the non-truncated induction
$\mathrm{Ind}_{W(A_3 \times A_2 \times A_1 \times A_0)}^{W(B_{10})}(\mathrm{triv})$ has partition of the form
\begin{equation}  \label{eq-LR}
(\ell_4 \odot \ell_3 \odot \ell_2 \odot \ell_1)^T \times (r_4 \odot r_3 \odot r_2 \odot r_1)^T,
\end{equation}
where $\ell_p + r_p = p$, and $\displaystyle \pi \odot \rho := \bigoplus_{\nu} c_{\pi\rho}^{\nu}\ \nu$ (here $c_{\pi\rho}^{\nu}$ is the Littlewood-Richardson coefficient, and the partitions $\pi$, $\rho$, $\nu$ are labeled in terms of rows). Then one can easily check that except $31 \times 42$, none of the partitions in $V^L(w_L)$ appears in \eqref{eq-LR}, and hence $V^L(w_L) \cap V^L(w_R) = \{31 \times 42\}$.

Indeed, one can compute $V^L(w_R)$ explicitly (this method works for $W_\frac12 = W(B_{n/2})$ only): Let $\mathcal{S} = (9,5,5,1,1)$ be a nilpotent orbit in $\mathfrak{so}_{21}(\mathbb{C})$. Similar to the $V^L(w_L)$ case, one has
\begin{align*}
V^L(w_R) &= \{\mathrm{Spr}(\mathcal{S})=\mathrm{Spr}((9,5,5,1,1)),\ \mathrm{Spr}((9,6,4,1,1)),\ \mathrm{Spr}((9,5,5,2,0)),\ \mathrm{Spr}((9,6,4,2,0))\}\\
&= \{31 \times 42, 21 \times 43, 3 \times 421, 2 \times 431\},
\end{align*}
all of which satisfy Equation \eqref{eq-LR} (for instance, $2 \times 431$ can be obtained from $(1 \odot 1 \odot 0 \odot 0)^T \times (3 \odot 2 \odot 2 \odot 1)^T$), and the only representation in common is $31 \times 42$.
\end{example}

We are now in the position to prove Theorem \ref{thm:unique}:

\bigskip
\noindent {\it Proof of Theorem \ref{thm:unique}.} By the first paragraph of the proof of \cite[Theorem 5.1]{MG94} again, the annihilator varieties
$\mathcal{V}(\mathrm{LAnn}(J(w,w')))$ and $ \mathcal{V}(\mathrm{RAnn}(J(w,w')))$ can be obtained as follows:
\begin{enumerate}
    \item The Weyl group representation 
    $$\sigma_L(w,w') := j_{W_{\frac12} \times W_{\frac12}}^{W} \Big (\mathrm{Spr}(\mathcal O_L (w)) \otimes \mathrm{Spr}(\mathcal O_L (w') \Big )$$ 
    is Springer,
    where $\mathcal O_L (w)$ and $\mathcal O_L(w')$ are defined by 
    $$ \mathcal{V}(\mathrm{LAnn}(J(w))=\overline{\mathcal O_L (w)}\  \text{ and  }\ 
     \mathcal{V}(\mathrm{LAnn}(J(w'))=\overline{\mathcal O_L (w')};$$
    \item $\mathcal{V}(\mathrm{LAnn}(J(w,w')))$ is the closure of the nilpotent orbit $\mathrm{Spr}^{-1}\left(\sigma_L(w,w')\right)$.
\end{enumerate}
(and analogous statements hold  for $\mathcal{V}(\mathrm{RAnn}(J(w,w')))$). By Proposition \ref{prop:uniquehalf}, the above Weyl group representation has highest generic degree ($\Leftrightarrow$ the orbit has smallest GK dimension) if and only if $w = w' = e$. The corresponding nilpotent orbit is
\begin{center}
$\mathrm{Spr}^{-1}\left(\sigma_{\bullet}(e,e)\right)$ ($\bullet = L$ or $R$),
\end{center} 
where $\sigma_{\bullet}(e,e) = j_{W_{\frac12} \times W_{\frac12}}^{W}(\mathrm{Spr}(\mathcal{P}) \otimes \mathrm{Spr}(\mathcal{P}))$ and 
$\mathcal{P}$ is as given in Lemma \ref{lem-P}. By Remark \ref{rmk:wh} and $j$-induction in stages, 
$$\sigma_{\bullet}(e,e) = j_{W_{\frac12}^{\flat} \times W_{\frac12}^{\flat}}^{W}(\mathrm{sgn}).$$ 
Then the result follows directly from Proposition \ref{prop-whw'}. \qed 

\begin{example} \label{eg:oprime}
    Continue with Example \ref{eg-P} above. Then $\mathcal{P} = (9,5,5,1,1)$ and $\mathrm{Spr}(\mathcal{P}) = 31 \times 42$. Therefore, 
    $$\Phi :=j_{W(B_{10}) \times W(B_{10})}^{W(B_{20})}(\mathrm{Spr}(\mathcal{P}) \otimes \mathrm{Spr}(\mathcal{P})) = 3311 \times 4422,$$ 
    and one can easily verify that $\Phi \in W(B_{20})^{\wedge}$ is a Springer representation with $$\mathrm{Spr}^{-1}(\Phi) = (9,8,6,5,5,4,2,1,1) = \mathcal{O}'.$$
\end{example}

\subsection{Proof of Theorem \ref{thm:main}} \label{sec:main}
Let $\Phi =  \mathrm{Spr}(\mathcal{O}') \in \widehat{W}$ be the Springer representation corresponding to $\mathcal{O}'$. Recall from Proposition \ref{prop-whw'} and Remark \ref{rmk:wh} that
$$\Phi = j_{W_{\frac12} \times W_{\frac12}}^W(\mathrm{Spr}(\mathcal{P}) \otimes \mathrm{Spr}(\mathcal{P})) =  j_{W_{\frac12}^{\flat} \times W_{\frac12}^{\flat}}^W(\mathrm{sgn} \otimes \mathrm{sgn}) = j_{W^{\flat}}^{W}(\mathrm{sgn}),$$
where $\mathcal{P}$ is as given in Equation \eqref{eq-Pdef}, and $W^{\flat} \leq W$ is the Weyl subgroup appearing in the character formula of $\mathrm{Lift}(R_{\emptyset})$ in Proposition \ref{prop:lift} and Example \ref{eg:161284}.

By projecting $X(e,e) \in \mathcal{B}$ in the coherent continuation representation to its $(\mathrm{Spr}(\mathcal{P}) \otimes \mathrm{Spr}(\mathcal{P}))$-isotypic component:
$$\mathrm{pr}_{\mathrm{Spr}(\mathcal{P}) \otimes \mathrm{Spr}(\mathcal{P})}(X(e,e)) := \sum_{w \in W_{\frac12} \times W_{\frac12}} \mathrm{tr}_{\mathrm{Spr}(\mathcal{P}) \otimes \mathrm{Spr}(\mathcal{P})}(w) (w \cdot X(e,e))$$
is a virtual representation consisting of irreducible representations with associated variety equal to the closure of $\mathcal{O}' = \mathrm{Spr}^{-1}\left(j_{W_{\frac12} \times W_{\frac12}}^{W}(\mathrm{Spr}(\mathcal{P}) \otimes \mathrm{Spr}(\mathcal{P}))\right) = \mathrm{Spr}^{-1}(\Phi)$. By \cite[Proposition 6.6]{BV85}, the right hand side of the above formula can be reduced to 
$$\sum_{w \in W_{\frac12}^{\flat} \times W_{\frac12}^{\flat}} \mathrm{sgn}(w) (w \cdot X(e,e)) = \sum_{w \in W^{\flat}} \mathrm{sgn}(w) (w \cdot X(e,e)) $$
which is precisely the character formula of $\mathrm{Lift}(R_{\emptyset})$ in Proposition \ref{prop:lift}. In other words, $\mathrm{Lift}(R_{\emptyset})$ is a linear combination of genuine unipotent representations $\widetilde{\mathcal{U}}  (\mathcal{O}')$. Then the Theorem \ref{thm:unique} implies that
$$\mathrm{Lift}(R_{\emptyset}) = \begin{cases} a\widetilde{\pi}(\mathcal{O}') &\textrm{for}\ \widetilde{G'} = \mathrm{Spin}_{2n+1}(\mathbb{C}) \\
a\widetilde{\pi}(\mathcal{O}') + b \widetilde{\pi}(\mathcal{O}')^{\vee}&\textrm{for}\ \widetilde{G'} = \mathrm{Spin}_{2n}(\mathbb{C}) \end{cases}$$
for some $a, b \in \mathbb{Z}$. Then the theorem follows by comparing the lowest $K$-type multiplicities on both sides of the equation.

\section*{Acknowledgments}
This project was initiated at the Conference ``Representation Theory XVIII'' in Dubrovnik, June 2023. The authors would like to thank the organizers for their hospitality. Tsai is supported by the National Science and Technology Council of Taiwan under grant no. 114-2115-M-008-003-MY3.  Wong is supported by the National Natural Science Foundation of China (no. 12341101, 12371033).

\end{document}